\documentclass[english, 12pt,leqno]{amsart}

\usepackage{a4,
amsbsy,amsfonts,amsgen,amsmath,amsopn,amssymb,amsthm,amstext,
enumerate,enumitem,exscale,stmaryrd,epsfig,
graphicx,mathrsfs,relsize,url}
\usepackage[pagebackref=true]{hyperref}
\usepackage{color}

\usepackage[utf8]{inputenc}
\usepackage[T1]{fontenc}

\newtheorem{thm}{Theorem}[section]
\newtheorem{cor}[thm]{Corollary}
\newtheorem{lem}[thm]{Lemma}
\newtheorem{prop}[thm]{Proposition}
\newtheorem{obs}[thm]{Observation}

\theoremstyle{definition}
\newtheorem{exe}[thm]{Example}
\newtheorem{rem}[thm]{Remark}
\newtheorem{numer}[thm]{Numerical note}
\newtheorem{ques}[thm]{Question}

\newcommand{\F}{\mathbf F}
\newcommand{\N}{\mathbf N}
\newcommand{\Z}{\mathbf Z}
\newcommand{\Q}{\mathbf Q}

\newcommand{\ko}{\mathbf k}
\newcommand{\T}{\mathbf T}
\newcommand{\Fnil}{\operatorname{Fnil}}
\newcommand{\Fsol}{\operatorname{Fsol}}
\newcommand{\W}{\mathrm W}
\newcommand{\FC}{\mathrm{FC}}
\newcommand{\Sub}{\mathrm{Sub}}

\newcommand{\llag}{\langle\!\langle}
\newcommand{\rrag}{\rangle\!\rangle} 

\newcommand{\id}{\mathrm{id}}
\newcommand{\Ker}{\mathrm{Ker}}
\newcommand{\End}{\mathrm{End}}

\newcommand{\MS}{\operatorname{MS}}
\newcommand{\MA}{\operatorname{MA}}
\newcommand{\MH}{\operatorname{MH}}
\newcommand{\Soc}{\operatorname{Soc}}
\newcommand{\SocA}{\operatorname{SocA}}
\newcommand{\SocH}{\operatorname{SocH}}

\title[Groups with irreducibly unfaithful subsets]
{Groups with irreducibly unfaithful subsets
\\
for unitary representations}



\date{October 16, 2019}

\author{Pierre-Emmanuel Caprace and Pierre de la Harpe}

\thanks{P.-E.C.\ is a F.R.S.-FNRS senior research associate.
This work was started at the Bernoulli Center of EPFL,
during the project \textit{Descriptive set theory and Polish groups};
the authors are grateful for support and hospitality}

\address{Pierre-Emmanuel Caprace:
UCLouvain -- IRMP,
Chemin du Cyclotron 2, box L7.01.02,
B-1348 Louvain-la-Neuve.}
\email{pe.caprace@uclouvain.be}

\address{Pierre de la Harpe:
Section de math\'ematiques, 
Universit\'e de Gen\`eve,
C.P.~64, 
CH--1211 Gen\`eve 4.}
\email{Pierre.delaHarpe@unige.ch}

\begin{document}

\maketitle

\begin{abstract}
Let $G$ be a group. A subset $F \subset G$ is called irreducibly faithful
if there exists an irreducible unitary representation $\pi$ of $G$
such that $\pi(x) \neq \id$ for all $x \in F \smallsetminus \{e\}$.
Otherwise $F$ is called irreducibly unfaithful.
Given a positive integer $n$, we say that $G$ has Property $P(n)$
if every subset of size $n$ is irreducibly faithful. 
Every group has $P(1)$, by a classical result of Gelfand and Raikov.
Walter proved that every group has $P(2)$.
It is easy to see that some groups do not have~$P(3)$.
\par

We provide a complete description of the irreducibly unfaithful subsets of size~$n$
in a countable group $G$ (finite or infinite) with Property $P(n-1)$:
it turns out that such a subset is contained
in a finite elementary abelian normal subgroup of $G$ of a particular kind. 
We deduce a characterization of Property~$P(n)$ purely in terms of the group structure.
It follows that, if a countable group $G$ has $P(n-1)$ and does not have $P(n)$,
then $n$ is the cardinality of a projective space over a finite field.
\par

A group $G$ has Property $Q(n)$ if,
for every subset $F \subset G$ of size at most $n$,
there exists an irreducible unitary representation $\pi$ of $G$
such that $\pi(x) \ne \pi(y)$ for any distinct $x, y$ in $F$.
Every group has $Q(2)$.
For countable groups,
it is shown that Property $Q(3)$ is equivalent to $P(3)$,
Property $Q(4)$ to $P(6)$, and Property $Q(5)$ to $P(9)$.
For $m, n \ge 4$, the relation between Properties $P(m)$ and $Q(n)$
is closely related to a well-documented open problem in additive combinatorics.
\end{abstract}


\setcounter{tocdepth}{1}

\tableofcontents
%
%
%
%
%
%
%
%
%

\begin{flushright}
\begin{minipage}{0.60\linewidth}
\itshape
Fid\`ele, infid\`ele ?\\
Qu'est-ce que \c ca fait,\\
Au fait ?\\
Puisque toujours dispose \`a couronner mon z\`ele\\
Ta beaut\'e sert de gage \`a mon plus cher souhait.\\
\smallskip
\hfill \upshape (Paul Verlaine, \emph{Chansons pour elle}, 1891.)
\end{minipage}
\end{flushright}

\vskip.2cm

\section{Introduction}
\label{Section:intro}

\subsection{Irreducibly unfaithful subsets}\label{subsec:1.a}
A subset $F$ of a group $G$ is called
\textbf{irreducibly faithful} if there exists
an irreducible unitary representation $\pi$ of $G$ in a Hilbert space $\mathcal H$
such that $\pi(x) \ne \id$ for all $x \in F$ with $x \ne e$.
(We denote by~$e$ the identity element of the group,
and by $\id$ the identity operator on the space~$\mathcal H$.)
Otherwise $F$ is called \textbf{irreducibly unfaithful}.
For $n \ge 1$, we say that $G$ has \textbf{Property}~$\mathbf{P(n)}$
if every subset of size at most $n$ is irreducibly faithful.
\par

\textit{Every group has Property $P(1)$:}
this is the particular case for discrete groups of a foundational result established
for all locally compact groups and continuous unitary representations
by Gelfand and Raikov~\cite{GeRa--43}
(see also the exposition in \cite[13.6.6]{Dixm--69},
and another proof for second-countable locally compact groups
in \cite[Pages 109--110]{Mack--76}).
\par

The following refinement of the Gelfand--Raikov Theorem is due to Walter: 
\textit{Every group has Property $P(2)$.}
In other words, \textit{in a group, every couple is irreducibly faithful} (!).
See \cite[Proposition 2]{Walt--74}, as well as \cite{Sasv--91} and \cite[1.8.7]{Sasv--95}.
\par

It is clear that Property $P(3)$ does not hold for all groups.
Indeed, Klein's Vierergruppe,
the direct product $C_2 \times C_2$ of two copies of the group of order~$2$,
does not have $P(3)$. 
\par
 
The first goal of this article is to characterize groups with $P(n)$ for all $n \ge 3$.
We focus on \emph{countable groups}, i.e.,
groups that are either finite or countably infinite.
What follows can be seen as a quantitative refinement of results in \cite{BeHa--08},
quoted in Theorem \ref{thm:Gasch} below.

\vskip.2cm

Before stating our main result, we need the following preliminaries.
Let $\ko$ be a finite field of order $q$;
in case $q = p$ is a prime, we write $\F_p$ instead of $\ko$.
For a group $G$, we denote by $\ko[G]$ its group algebra over $\ko$.
We recall that any abelian group $U$ whose exponent is a prime $p$
carries the structure of a vector space over~$\F_p$,
which is invariant under all group automorphisms of $U$.
In other words, the group structure on $U$ canonically determines an $\F_p$-linear structure.
In particular, an abelian normal subgroup $U$ of exponent $p$ in a group $G$
may be viewed, in a canonical way, as an $\F_p[G]$-module.
Moreover, $U$ is minimal as a normal subgroup of $G$
if and only if $U$ is simple as an $\F_p[G]$-module.
(We rather use $W$ instead of $U$ when such a simple module appears below,
and $V$ for direct sums of particular numbers of copies of simple modules.)
\par

Let $G$ be a group and $U$ an $\F_pG]$-module.
The \textbf{centralizer} of $U$ is the $\F_p$-algebra
$$
\mathcal L_{\F_p[G]}(U) =
\{\alpha \in \End_{\F_p}(U) \mid g.\alpha (u) = \alpha(g.u)
\hskip.2cm \text{for all} \hskip.2cm
g \in G, \ u \in U\} .
$$
If $W$ is a simple $\F_p[G]$-module, Schur's lemma ensures
that $\mathcal L_{\F_p[G]}(W)$ is a division algebra over $\F_p$
\cite[\S~4, Proposition~2]{Bo--A8}.
If in addition $W$ is finite,
then the algebra $\mathcal L_{\F_p[G]}(W)$ is a finite field extension of $\F_p$,
by Wedderburn's Theorem \cite[\S~11, no.~1]{Bo--A8}.
In this case, $W$ is a vector space over $\mathcal L_{\F_p[G]}(W)$,
and the dimension of $W$ over $\mathcal L_{\F_p[G]}(W)$
is the quotient of $\dim_{\F_p}(W)$
by the degree of the extension $\mathcal L_{\F_p[G]}(W)$ of $ \F_p$.
\par

For example, consider a finite field extension $\ko$ of $\F_p$,
a positive integer $m$, 
the vector space $W = \ko^m$,
and the general linear group $\mathrm{GL}(W) = \mathrm{GL}_m(\ko)$
together with its natural action on $W$.
View $W$ as a vector space over $\F_p$, and as an $\F_p\mathrm{GL}(W)]$-module.
Then $\mathcal L_{\F_p[\mathrm{GL}(W)]}(W)$
and $\ko$ can be identified,
and $\dim_{\ko}(W) = m$.
\par

Our main result reads as follows; the proof is in Section~\ref{Section:proofof1.1}.

\begin{thm}
\label{thm:CharP(n)}
Let $G$ be a countable group and $n$ a positive integer.
The following assertions are equivalent. 
\begin{enumerate}[label=(\arabic*)]
\item\label{1DEthm:CharP(n)}
$G$ does not have $P(n)$.
\item\label{2DEthm:CharP(n)}
There exist a prime $p$,
a finite normal subgroup $V$ in $G$
which is an elementary abelian $p$-group,
and a finite simple $\F_p[G]$-module $W$,
such that the following properties hold,
where $\ko$ denotes the centralizer field $\mathcal L_{\F_p[G]}(W)$,
 $m = \dim_{\ko}(W)$, and $q = \vert \ko \vert$:
\begin{enumerate}[label=(\roman*)]
\item
$V$ is isomorphic to the direct sum of $m+1$ copies of $W$, as an $\F_pG]$-module;
\item
$q$ is a power of $p$ and
$q^m + q^{m-1} + \dots + q + 1 \le n$.
\end{enumerate}
\end{enumerate}
\end{thm}

Notice that the inequality $q^m + q^{m-1} + \dots + q + 1 \ge 3$ always holds,
since $q \ge p \ge 2$ and $m \ge 1$.
Therefore, in the particular cases of $n = 1$ and $n = 2$,
Theorem \ref{thm:CharP(n)} recovers for countable groups
0the results of Gelfand--Raikov and Walter quoted above.
In the particular case of $n = 3$, the theorem shows
that the only obstruction to $P(3)$ can be expressed
in terms of Klein's Vierergruppe:

\begin{cor}
\label{cor:P(3)}
A countable group has $P(3)$ if and only if
its centre does not contain any subgroup isomorphic to $C_2 \times C_2$.
\end{cor}

\begin{proof}
For $n=3$, we have $p = q = 2$ and $m = 1$
in \ref{2DEthm:CharP(n)} of Theorem \ref{thm:CharP(n)}.
Hence $V = W \oplus W$ is an $\F_2[G]$-module of dimension $2$ over $\F_2$.
Moreover, the action of $G$ is trivial on $W$,
therefore also on $V$, and this means that, as a normal subgroup of $G$,
the group $V$ is central.
\end{proof}

\begin{cor}
\label{finitenormalcentral}
Let $n$ be a positive integer and $G$ a countable group.
Assume that every minimal finite abelian normal subgroup of $G$ is central.
\par

Then $G$ does not have $P(n)$ if and only if
$G$ contains a central subgroup isomorphic to $C_p \times C_p$
for some prime $p \le n-1$.
\end{cor}

In the following proof, and later, we denote by $\T$
the group of complex numbers of modulus one.
Recall that, for any irreducible unitary
representation $\pi$ of a group $G$
with centre $Z$ on a Hilbert space $\mathcal H$, 
there exists by Schur's Lemma
a unitary character $\chi \hskip.1cm \colon Z \to \T$
such that $\pi(g) = \chi(g) \id$ for every $g \in Z$.

\begin{proof}
Suppose that $G$ does not have Property $P(n)$.
Let $p$ be a prime and $V, W$
as in Assertion \ref{2DEthm:CharP(n)} of Theorem \ref{thm:CharP(n)}.
Since the action by conjugation of $G$
on minimal finite abelian normal subgroups is trivial by hypothesis,
the $\F_p[G]$-module $W$, which is simple,
is of dimension $1$ over $\F_p$.
With the notation of Theorem \ref{thm:CharP(n)},
this implies that $m = 1$ and $\ko = \F_p$.
It follows that $V = W \oplus W \cong C_p \times C_p$.
\par

Conversely, if $G$ contains 
a central subgroup $V$ isomorphic to $C_p \times C_p$ for some prime $p \le n-1$,
consider a subset $F$ of $G$ of size~$p+1$
containing a generator of each of the $p+1$ non-trivial cyclic subgroups of $V$.
As recalled just before the present proof,
every irreducible unitary
representation of $G$ provides a unitary character
$\chi \hskip.1cm \colon C_p \times C_p \to \T$.
We have $F \cap \Ker \chi \not \subset \{e\}$,
hence $F$ is irreducibly unfaithful.
\end{proof}

\begin{exe}
\label{examplenormalcentral}

There are several classes of groups which have the property that
``every minimal finite abelian
normal subgroup is central'':

\vskip.2cm

(1)
Torsion-free groups have the property.

\vskip.2cm

(2)
Icc-groups, that is
infinite groups in which all conjugacy classes distinct from $\{e\}$ are infinite,
have the property.

\vskip.2cm

(3)
A group $G$ without non-trivial finite quotient has the property.
Indeed, if $N$ is a finite normal subgroup of $G$,
the action by conjugation of $G$ on $N$ provides
a homomorphism from $G$ to
the group of automorphisms of $N$;
since this homomorphism is trivial by hypothesis, $N$ is central.

\vskip.2cm

(4)
In a connected algebraic group $G$,
a Zariski-dense subgroup $\Gamma$
has the property.
To check this, it suffices to show that
the FC-centre $\FC(\Gamma)$ of $\Gamma$
coincides with the centre of $\Gamma$.
Recall that the \textbf{FC-centre} of a group
is the characterisitic subgroup consisting
of elements which have a finite conjugacy class;
the FC-centre of a group contains every finite normal subgroup.
\par

Recall that the centraliser of an element in an algebraic group is a Zariski-closed subgroup.
Given $\gamma \in \FC(\Gamma)$,
the centralizer $C_G(\gamma)$ is Zariski-closed and contains a finite index subgroup of $\Gamma$.
Since $\Gamma$ is Zariski-dense, it follows that $C_G(\gamma)$ is of finite index in $G$.
Therefore $C_G(\gamma) = G$ since $G$ is connected.
Hence $\gamma$ is in the centre $Z(G)$ of $G$.
It follows that $\gamma \in \Gamma \cap Z(G) = Z(\Gamma)$.
This shows that $\FC(\Gamma)$ is central in $\Gamma$,
and consequently that every finite normal subgroup of $\Gamma$ is central.
\par

In particular, this applies to all lattices in connected semi-simple groups
over non-discrete locally compact fields without compact factors, by the Borel Density Theorem.

\vskip.2cm

(5)
A nilpotent group has the property,
because any non-trivial normal subgroup of a nilpotent group
has a non-trivial intersection with the centre
\cite[Chap.~I, \S~6, no.~3]{Bo--A1-3}.
\end{exe}

\vskip.2cm

Theorem~\ref{thm:CharP(n)} also has the following immediate consequence:

\begin{cor}
\label{P(n-1)impliesP(n)}
Let $n$ be an integer, $n \ge 2$.
Suppose that there is no prime power $q$ and integer $m \ge 1$ such that 
$n = q^m + q^{m-1} + \dots + q + 1$. 
\par

Every countable group that has $P(n-1)$ also has $P(n)$. 
\end{cor}

When $n = q^m + q^{m-1} + \dots + q + 1$, we have the following.

\begin{exe}
\label{exampleGqm}
Consider a prime $p$,
a power $q$ of $p$,
an integer $m \ge 1$,
a field $\ko$ of order $q$, 
the vector space $W = \ko^m$,
and the group $\mathrm{GL}(W) = \mathrm{GL}_m(\ko)$.
Let $V_0, V_1, \dots, V_m$ be $m+1$ copies of $W$.
The group $\mathrm{GL}(W)$ acts diagonally on $V := \bigoplus_{i=0}^m V_i$.
Since $V$ is an elementary abelian $p$-group,
it can be viewed as an $\F_p \mathrm{GL}(W)]$-module.
Define the semi-direct product group
$$
G_{(q, m)} = \mathrm{GL}(W) \ltimes V.
$$
\par

Let $N$ be a normal subgroup of $G_{(q, m)}$.
Assume that $N \cap V = \{e\}$.
On the one hand, $N$ commutes with $V$, hence acts trivially on $V$ by conjugation.
On the other hand, the triviality of $N \cap V$ implies that
$N$ maps injectively in the quotient $G_{(q, m)}/V \cong \mathrm{GL}(W)$,
whose conjugation action on $V$ is faithful.
Hence $N = \{e\}$.
This shows that every non-trivial normal subgroup of $G_{(q, m)}$
has a non-trivial intersection with $V$.
In particular, every minimal normal subgroup of $G_{(q, m)}$ is contained in $V$,
and thus is abelian.
Hence it corresponds to a simple $\F_p[G_{(q, m)}]$-submodule of $V$.
Therefore, every minimal abelian normal subgroup of $G_{(q, m)}$
is isomorphic to $W$ as an $\F_pG_{(q, m)}]$-module.
\par

We now set $n = q^m + q^{m-1} + \dots + q + 1$.
Then Theorem~\ref{thm:CharP(n)} implies that $G_{(q, m)}$
has Property $P(n-1)$ but not $P(n)$.
\end{exe}

Notice that the group $G_{(q, 1)}$ is the semi-direct product
$\ko^\times \ltimes (\ko \oplus \ko)$, where $\ko$ is a field of order~$q$.
The group $G_{(2,1)}$ is Klein's Vierergruppe.
The group $G_{(3,1)}$ appears in \cite[Note F]{Burn--11}
as an example of a finite group with trivial centre
which does not admit any faithful irreducible representation.
The group $G_{(4,1)}$ appears in \cite[Problem 2.19]{Isaa--76}
for the same reason.
Our groups $G_{(q,1)}$ appear in the historical review section of \cite{Szec--16},
where they are denoted by $G(2,q)$.
The tables
$$
{
\begin{array}{c|ccccccccccccccccccc}
q & 2 & 3 & 4 & 5 & 7 & 8 & 9 & 11
\\
\hline
\vert G_{(q,1)} \vert & 4 & 18 & 48 & 100 & 294 & 448 & 649 & 1110
\\
\end{array}
}
\hskip.5cm \text{and} \hskip.5cm
{
\begin{array}{c|ccccccccccccccccccc}
q & 2 & 3
\\
\hline
\vert G_{(q,2)} \vert & 384 & 34992
\\
\end{array} 
}
$$
give the orders of the $8$ smallest groups $G_{(q,1)}$ 
and the $2$ smallest groups $G_{(q,2)}$.

\begin{numer}
The sequence of positive integers
which are of the form $q^m + q^{m-1} + ... + q + 1$
for some prime power $q$ and positive integer $m$
is Sequence A258777 of \cite{OEIS}; the first $25$ terms are
$$
3, 4, 5, 6, 7, 8, 9, 10, 12, 13, 14, 15, 17, 18, 20, 21, 24, 26, 28, 30, 31, 32, 33, 38, 40
$$
(note that we start with $3$ whereas A258777 starts with $1$).
The first $10 \hskip.1cm 000$ terms appear on
\hskip.2cm
\url{https://oeis.org/A258777/b258777.txt}
\hskip.2cm
where the last term is $101 \hskip.1cm 808$.
For terms below $100$, the largest gap is between the $45$th term and the $46$th term,
i.e., between $91$ and $98$;
it follows from Corollary \ref{P(n-1)impliesP(n)}
that a group with Property $P(91)$ has necessarily Property $P(97)$.
\par

It is a consequence of the Prime Number Theorem that the asymptotic density
of this sequence is $0$.
In other words, if for $k \ge 1$ we denote by $R(k)$
the number of positive integers less than $k$ which are terms of this sequence,
then $\lim_{k \to \infty} R(k)/k = 0$.
See \cite[Appendix B]{Radu--17}.
\par

Note also that the $21$st term, which is $31$,
can be written in two ways justifying its presence in the sequence:
$31 = 2^4 + 2^3 + 2^2 + 2 + 1 = 5^2 + 5 + 1$.
It is a conjecture that there are no other terms with this property,
but this is still open.
Indeed, conjecturally, the Goormaghtigh equation
$$
\frac{x^M - 1}{x-1} = \frac{y^N-1}{y-1}
$$
has no solution in integers $x,y,M,N$ such that
$x,y \ge 2$, $x \ne y$, and $M, N \ge 3$,
except $31 = \frac{2^5-1}{2-1} = \frac{5^2-1}{5-1}$
and $8191 = \frac{2^{13}-1}{2-1} = \frac{90^3-1}{90-1}$.
We are grateful to Emmanuel Kowalski and Yann Bugeaud
for information on the relevant literature,
which includes \cite{Rafa--16, Goor--17, BuSh--02, He--09}.
\end{numer}

In a group which has $P(n-1)$ and not $P(n)$,
irreducibly unfaithful subsets of size $n$
are contained in finite normal subgroups of a very particular kind,
described in Theorem \ref{thm:UnfaithfulSetSize-n}.
Here is a partial statement of this theorem:

\begin{prop}
\label{partial4.5}
Let $G$ be a countable group and $n$ a positive integer.
Assume that $G$ has Property $P(n-1)$.
Let $F$ be a finite subset of $G$ of size $n$ which is irreducibly unfaithful,
and let $U$ denote the smallest normal subgroup of $G$ containing $F$.
\par
Then there exists a prime $p$ such that $U$ is a finite elementary abelian $p$-group,
and $U$ is contained in the mini-socle $\MA (G)$
 (as defined in Subsection \ref{subsec:Gaschutz} below).
 \end{prop}
 
\subsection{Irreducibly faithful groups}

A group is \textbf{irreducibly faithful} if it has a faithful irreducible unitary representation. 
Clearly, an irreducibly faithful group $G$
 has $P(n)$ for all $n \ge 1$.
The problem of characterizing finite groups
which are irreducibly faithful
has been addressed by Burnside in \cite[Note F]{Burn--11},
where a sufficient condition is given.
Since then, various papers have been published on the subject,
providing various answers to Burnside's question;
see the historical review in \cite{Szec--16}.
\par

Gasch\"utz~\cite{Gasc--54} obtained a short proof of the following simple criterion:
\emph{a finite group $G$ admits a faithful irreducible representation
over an algebraically closed field of characteristic~$0$
if and only if the abelian part of the socle of $G$ is generated by a single conjugacy class}.
For unitary representations,
this result has been extended
to the class of all countable groups in \cite[Theorem~2]{BeHa--08};
see Subection~\ref{subsec:Gaschutz} below. As a consequence of Theorem~\ref{thm:CharP(n)},
we shall obtain the following supplementary characterization.

\begin{cor}
\label{cor:P(n)-For-All-n}
For a countable group $G$, the following conditions are equivalent:
\begin{enumerate}[label=(\roman*)]
\item\label{iDEcor:P(n)-For-All-n}
$G$ has a faithful irreducible unitary representation. 
\item\label{iiDEcor:P(n)-For-All-n}
$G$ has $P(n)$ for all $n \ge 1$.
\item\label{iiiDEcor:P(n)-For-All-n}
For any prime $p$,
the group $G$ does not contain any finite abelian normal subgroup $V$ of expontent $p$
with the following properties:
there exists a finite simple $\F_p[G]$-module $W$,
with associated centralizer $\ko = \mathcal L_{\F_p[G]}(W)$
and dimension $m = \dim_{\ko}(W)$, 
such that $V$ is isomorphic as an $\F_pG]$-module
to the direct sum of $m+1$ copies of $W$.
\end{enumerate}
\end{cor}

In the case of finite groups, the equivalence between
\ref{iDEcor:P(n)-For-All-n} and \ref{iiDEcor:P(n)-For-All-n} is trivial,
while the equivalence between
\ref{iDEcor:P(n)-For-All-n} and \ref{iiiDEcor:P(n)-For-All-n}
is due to Akizuki (see \cite[Page 207]{Shod--31}).
\par

Let $G$ be a countable group
in which every minimal finite abelian normal subgroup is central
(see Corollary \ref{finitenormalcentral}).
For such a group, Condition \ref{iiiDEcor:P(n)-For-All-n} above
can be reformulated as follows.
\begin{enumerate}
\item[(iii$'$)]
The group $G$ does not contain any central subgroup isomorphic to $C_p \times C_p$
for some prime $p$.
\end{enumerate}
\par

For uncountable groups,
some of the equivalences of Corollary \ref{cor:P(n)-For-All-n} may fail.
See Remark \ref{remcountablegps}.

\subsection{Abelian groups}
Corollary \ref{finitenormalcentral} applies in particular to countable abelian groups.
The following Proposition \ref{prop:Abelian}
shows that the conclusion holds for all abelian groups, countable or not.
Our proof does not rely on Theorem~\ref{thm:CharP(n)},
but uses the following result of Mira Bhargava,
which is Theorem 4 of \cite{Bhar--02}.

\begin{thm}[Mira Bhargava]
\label{thm:Bhargava}
For any group $G$ and any natural number $n$, the following conditions are equivalent:
\begin{enumerate}[label=(\roman*)]
\item
$G$ is the union of $n$ proper normal subgroups. 
\item
$G$ has a quotient isomorphic to $C_p \times C_p$, for some prime $p \le n-1$.
\end{enumerate}
\end{thm}

\begin{prop}
\label{prop:Abelian}
Let $n$ be a positive integer and let $G$ be an abelian group.
\par

Then $G$ does not have $P(n)$
if and only if $G$ contains a subgroup isomorphic to $C_p \times C_p$
for some prime $p \le n-1$. 
\end{prop}

\begin{proof}
Assume that $G$ does not have Property $P(n)$.
Let $F \subset G \smallsetminus \{e\}$
be an irreducibly unfaithful subset of $G$ of size $\le n$.
Let $\widehat G$ be the Pontryagin dual of $G$,
namely the group of all unitary characters $G \to \T$.
For each $x \in F$,
let $H_x = \{\chi \in \widehat G \mid \chi(x) = 1\}$;
it is a subgroup of $\widehat G$.
Since $G$ has $P(1)$, we have $H_x \neq \widehat G$.
Since $F$ is irreducibly unfaithful we have $\widehat G = \bigcup_{x \in F} H_x$.
Since $\widehat G$ is abelian, every subgroup is normal,
and Theorem~\ref{thm:Bhargava} ensures that
$\widehat G$ maps onto $C_p \times C_p$,
for some prime $p \le \vert F \vert -1 \le n-1$.
By duality (see \cite[chap.\ II, \S~1, no~7, Th.~4]{Bo--TS}),
it follows that $G$ contains a subgroup isomorphic to the dual of $C_p \times C_p$,
i.e., a subgroup isomorphic to $C_p \times C_p$ itself.
\par

The proof of the converse implication is as in the proof of Corollary \ref{finitenormalcentral}.
\end{proof}

It is easy to characterize abelian groups having faithful unitary characters,
i.e., having faithful irreducible unitary representations.
We denote by $\mathfrak c$ the cardinality of the continuum, i.e., of~$\T$.

\begin{prop}
\label{almostinFuchs}
For an abelian group $G$, the following conditions are equivalent:
\begin{enumerate}[label=(\roman*)]
\item\label{iDEalmostinFuchs}
$G$ has a faithful unitary character, i.e.,
$G$ is isomorphic to a subgroup of~$\T$.
\item\label{iiDEalmostinFuchs}
The cardinality of $G$ is at most $\mathfrak c$,
and there is not any subgroup of $G$ isomorphic to $C_p \times C_p$,
for some prime $p$.
\end{enumerate}
\end{prop}

\begin{proof}
Let $A$ be an abelian group.
Denote by $r_0(A)$ the torsion-free rank and,
for each prime $p$, by $r_p(A)$ the $p$-rank of $A$.
For a prime $p$, denote by $\Z(p^\infty)$ the quasicyclic group $\Z[1/p] / \Z$.
\par

Let $G$ be an abelian group; denote by $E$ a \textbf{divisible hull} of $G$.
Recall the following standard results on $G$ and $E$.
First, as any divisible group, $E$ is isomorphic to a direct sum
$$
E \cong 
\Big( \bigoplus_{r_0(E)} \Q \Big) \mathlarger{\oplus} \Big( \bigoplus_p
\Big( \bigoplus_{r_p(E)} \Z(p^\infty) \Big) \Big) ,
\leqno{(*)}
$$
and we also have
$$
\T \cong 
\Big( \bigoplus_{\mathfrak c} \Q \Big) \mathlarger{\oplus} \Big( \bigoplus_p
\Z(p^\infty) \Big) ;
\leqno{(**)}
$$
see \cite[Theorem 23.1]{Fuch--70}.
Second, a divisible group (for example $\T$) has a subgroup isomorphic to $G$
if and only if it has a subgroup isomorphic to $E$
\cite[Theorem 24.4]{Fuch--70}.
Furthermore, we have
$$
r_0(E) = r_0(G)
\hskip.5cm \text{and} \hskip.5cm
r_p(E) = r_p(G) \hskip.2cm \text{for all prime $p$} ;
\leqno{(***)}
$$
see \cite[Section 24]{Fuch--70}.
\par

It follows that $G$ satisfies Condition \ref{iDEalmostinFuchs}
if and only if $r_0(G) \le \mathfrak c$
and $r_p(G) \le 1$ for all primes $p$,
that is if and only if $G$ satisfies Condition \ref{iiDEalmostinFuchs}
\end{proof}

\begin{rem}
\label{remcountablegps}
Corollary \ref{cor:P(n)-For-All-n} does not extend
to groups of cardinality larger than~$\mathfrak c$.
Indeed, by Proposition \ref{prop:Abelian},
any torsion-free abelian group has $P(n)$ for all $n \ge 0$
(Condition \ref{iiDEcor:P(n)-For-All-n} of Corollary \ref{cor:P(n)-For-All-n}),
but cannot be isomorphic to a subgroup of $\T$
when its cardinality is larger than $\mathfrak c$
(negation of Property \ref{iDEcor:P(n)-For-All-n} of Corollary \ref{cor:P(n)-For-All-n}).
\par

We have not been able to decide whether
Conditions \ref{iDEcor:P(n)-For-All-n} and \ref{iiDEcor:P(n)-For-All-n}
of Corollary \ref{cor:P(n)-For-All-n}
are equivalent for all groups of cardinality at most $\mathfrak c$.
They are for abelian groups of cardinality at most $\mathfrak c$,
this is Proposition \ref{almostinFuchs}.
\par

Conditions \ref{iiDEcor:P(n)-For-All-n} and \ref{iiiDEcor:P(n)-For-All-n}
of Corollary \ref{cor:P(n)-For-All-n}
are equivalent for any abelian group,
this is Proposition \ref{prop:Abelian}.
\end{rem}

\subsection{Irreducible versus factor representations}

Recall that two unitary representations $\pi$, $\pi'$ of a group $G$
are called \textbf{disjoint} if there do not exist non-zero subrepresentations
$\rho$ of $\pi$ and $\rho'$ of $\pi'$ which are equivalent.
A unitary representation $\pi$ of a group $G$
is called a \textbf{factor representation} (or a \textbf{primary representation})
if it cannot be decomposed as the direct sum of two disjoint subrepresentations.
Equivalently the unitary representation $\pi$ is a factor representation
if the von Neumann algebra generated by $\pi(G)$ is a factor.
Every irreducible unitary representation is a factor representation.
The direct sum of several copies of a given irreducible unitary representation
is an example of a factor representation which is not irreducible.
However, some factor representations do not contain any irreducible subrepresentations.
The notion of factor representation plays a key role
in the theory of unitary representations on infinite-dimensional Hilbert spaces,
see \cite{Dixm--69} and \cite{Mack--76}.
We record here the following observation,
which implies that the results above remain unchanged
if one replaces the class of irreducible unitary representations
by the larger class of factor representations
(proof in Subsection \ref{subsec:SubsidiaryBeHa}).

\begin{prop}
\label{prop:Factor->Irrep}
Let $G$ be a countable group. For any factor representation $\pi$ of $G$,
there is an irreducible unitary representation $\sigma$ of $G$
such that $\Ker(\sigma) = \Ker(\pi)$. 	
\end{prop}	

In particular, a countable group is irreducibly faithful if and only if it is factorially faithful.

\subsection{Irreducibly injective subsets}
\label{IntroQ(n)}

A natural variation on the notion of irreducible (un)faith\-fulness can be defined as follows.
\par

A subset $F$ of a group $G$ is called \textbf{irreducibly injective}
if $G$ has an irreducible unitary representation $\pi$
such that the restriction $\pi \vert_F$ of $\pi$ to $F$ is injective.
We say that $G$ has Property $Q(n)$
if every subset of $G$ of size~$\le n$ is irreducibly injective.
It is a tautology that every group has Property $Q(1)$.
Though we do not know a characterization of countable groups which have Property $Q(n)$ for $n \ge 6$
in terms of the group structure,
some of the results we show in Section \ref{SectionQ(n)} can be summarized as follows.

\begin{prop}
\label{summingupSection6}
Let $G$ be a countable group and $n$ a positive integer.
\begin{enumerate}[label=(\roman*)]
\item\label{iDEsummingupSection6}
If $G$ has $P\left( \binom{n}{2} \right)$, then $G$ has $Q(n)$;
in particular, every countable group has $Q(2)$.
\item\label{iiDEsummingupSection6}
If $G$ has $Q(n)$, then $G$ has $P(n)$.
\item\label{iiiDEsummingupSection6}
$G$ has $Q(3)$ if and only if $G$ has $P(3)$.
\item\label{ivDEsummingupSection6}
$G$ has $Q(4)$ if and only if $G$ has $P(6)$.
\item\label{vDEsummingupSection6}
$G$ has $Q(5)$ if and only if $G$ has $P(9)$.
\end{enumerate}
\end{prop}

Understanding Property $Q(n)$ for larger $n$
is closely related to a well-documented open problem
in additive combinatorics.
See Subsection \ref{AdditiveComb}.

\vskip.2cm

We are grateful to Yves Cornulier for his comments
on a previous version of our text.

\section{Irreducibly faithful groups and related facts}

\subsection{Feet, mini-feet and Gasch\"utz' Theorem}
\label{subsec:Gaschutz}

Theorem \ref{thm:Gasch} below is due to Gasch\"utz in the case of finite groups
\cite{Gasc--54} (see also \cite[Theorem 42.7]{Hupp--98}),
and has been generalized to countable groups in \cite[part of Theorem 2]{BeHa--08}.
First we remind some terminology.
\par

In a group $G$, a \textbf{mini-foot} is a minimal non-trivial finite normal subgroup;
we denote by $\mathcal M_G$ the set of all mini-feet of $G$.
The \textbf{mini-socle} of $G$ is the subgroup $\MS (G)$
generated by $\bigcup_{M \in \mathcal M_G} M$;
the mini-socle is $\{e\}$ if $\mathcal M_G$ is empty,
for example $\MS (\Z) = \{0 \}$.
\par

Let $\mathcal A_G$ denote the subset of $\mathcal M_G$
of abelian mini-feet, and $\mathcal H_G$
the complement of $\mathcal A_G$ in $\mathcal M_G$.
The \textbf{abelian mini-socle} of $G$ is the subgroup $\MA (G)$
generated by $\bigcup_{A \in \mathcal A_G} A$,
and the \textbf{semi-simple part} $\MH (G)$ of the mini-socle
is the subgroup generated by $\bigcup_{H \in \mathcal H (G)} H$.
For examples of $\MA (G)$, see \ref{exe:socletrunk} below.
\par

In the context of finite groups,
mini-foot and mini-socle are respectively called \textbf{foot} and \textbf{socle}.
We denote the socle of a finite group $G$ by $\Soc(G)$,
the abelian socle by $\SocA(G)$,
and the semi-simple part of the socle by $\SocH(G)$.
The structure of the socle is due to Remak \cite{Rema--30}.
\par

For general groups, finite or not, the structure of the mini-socle
can be described similarly, as follows.
We write $\prod'_{\iota \in I} G_\iota$ for the restricted sum
of a family of groups $(G_\iota)_{\iota \in I}$;
recall that it is the subgroup of the direct product
consisting of elements $(g_\iota)_{\iota \in I} \in \prod_{\iota \in I} G_\iota$
such that $g_\iota$ is the identity of $G_\iota$ for all but finitely many $\iota \in I$.

\begin{prop}
\label{structureminisocle}
Let $G$ be a group.
Let $\mathcal M_G, \MS(G), \mathcal A_G, \MA(G), \mathcal H_G, \MH(G)$
be as above.
\begin{enumerate}[label=(\arabic*)]
\item\label{1DEstructureminisocle}
Every abelian mini-foot $A$ in $\mathcal A_G$ is an elementary abelian $p$-group $(\F_p)^m$
for some prime $p$ and positive integer $m$.
\item\label{2DEstructureminisocle}
There exists a subset $\mathcal A'_G$ of $\mathcal A_G$
such that $\MA (G) = \prod'_{A \in \mathcal A'_G} A$.
In particular $\MA (G)$ is abelian.
\item\label{3DEstructureminisocle}
Every non-abelian mini-foot $H$ in $\mathcal H_G$
is a direct product of a finite number of isomorphic non-abelian simple groups,
conjugate with each other in $G$.
\item\label{4DEstructureminisocle}
$\MH (G)$ is the restricted sum $\prod'_{H \in \mathcal H_G} H$
of the mini-feet in $\mathcal H_G$.
\item\label{5DEstructureminisocle}
$\MS (G)$ is the direct product $\MA (G) \times \MH (G)$.
\item\label{6DEstructureminisocle}
Each of the subgroups $\MS (G)$, $\MA (G)$, $\MH (G)$ is characteristic
(in particular normal) in $G$.
\item\label{7DEstructureminisocle}
Let $r \hskip.1cm \colon G \twoheadrightarrow Q$
be a surjective homomorphism of $G$ onto a group $Q$.
Then, for every mini-foot $X$ of $G$, either $r(X)$ is trivial or $r(X)$ is a mini-foot of $Q$.
In particular $r$ maps $\MA(G)$ [respectively $\MH(G)$, $\MS(G)$]
to a subgroup of $\MA(Q)$ [respectively $\MH(Q)$, $\MS(Q)$] which is normal in $Q$. 
\end{enumerate}
\end{prop}

We refer to \cite[Proposition 1]{BeHa--08} for the proof.
\par

The next result is a slight reformulation of the equivalence between (i) and (iv)
in \cite[Theorem~2]{BeHa--08}.

\begin{thm}
\label{thm:Gasch}
For a countable group $G$, the following assertions are equivalent. 
\begin{enumerate}[label=(\roman*)]
\item
$G$ has a faithful irreducible unitary representation.
\item
Every finite normal subgroup of $G$ contained in the abelian mini-socle
is generated by a single conjugacy class. 
\end{enumerate}
\end{thm}

This result is a crucial tool for the proof of Theorem~\ref{thm:CharP(n)}.
Moreover, we shall also need subsidiary facts that we will extract from~\cite{BeHa--08}.
They will be presented in Section~\ref{subsec:SubsidiaryBeHa} below.

\subsection{On characteristic subgroups that are directed unions of finite normal subgroups}

The next lemma, whose straightforward proof is left to the reader,
ensures that every group has a unique largest normal subgroup
that is the directed union of all its finite normal subgroups
[respectively all its soluble finite normal subgroups].
In particular these subgroups are characteristic.
\par

For an element $g \in G$ and a subset $F \subset G$,
we denote by $\llag g \rrag_G$ the normal subgroup of $G$ generated by $\{g\}$,
and by $\llag F \rrag_G$ that generated by $F$.

\begin{lem}
\label{lem:N1Nk}
Let $G$ be a group.
\begin{enumerate}[label=(\roman*)]
\item\label{iDElem:N1Nk}
Let $k$ be a positive integer and $N_1, \hdots, N_k$ finite normal subgroups of $G$.
The subgroup of $G$ generated by $\bigcup_{j=1}^k N_j$ is the product $N_1N_2 \dots N_k$,
in particular it is a finite normal subgroup of $G$.
\item\label{iiDElem:N1Nk}
The subset 
$$
\W(G) = \{ g \in G \mid \text{the normal subgroup $\llag g \rrag_G$ is finite} \} 
$$
is a characteristic subgroup of $G$, and is 
the directed union of all finite normal subgroups of $G$.
\item\label{iiiDElem:N1Nk}
The subset
$$
\hskip1cm
\Fsol (G) = \{ g \in G \mid \text{the normal subgroup $\llag g \rrag_G$ is finite and soluble} \}
$$
is a characteristic subgroup of $G$, and is
the directed union of all soluble finite normal subgroups of $G$.
\end{enumerate}
\end{lem}

The FC-centre $\FC(G)$ has been defined in Example \ref{examplenormalcentral}.
The characteristic subgroup $\W (G)$ is the \textbf{torsion FC-centre} of $G$.
According to Dicman's Lemma \cite[14.5.7]{Robi--96},
which ensures that every element of finite order in the FC-centre of $G$
has a finite normal closure,
$\W (G)$ is also the set of elements of finite order in $\FC(G)$.
The inclusions
$$
\MA(G) \le \Fsol (G) \le \W(G) \le \FC(G)
$$
and 
$$
\MS(G) \le \W(G)
$$
follow from the definitions.

\vskip.2cm

We illustrate those notions by discussing several examples.

\begin{exe}[Abelian mini-socles and other charactersitic subgroups]
\label{exe:socletrunk}

(1)
Let $p$ be a prime. In the cyclic group $\Z / p^2\Z$,
the abelian socle $\Z / p\Z$ (which is also the socle)
is a proper subgroup of $\Fsol (\Z / p^2\Z) = \Z / p^2\Z$.
\par

More generally, for a torsion abelian group $G$,
the abelian mini-socle (which is also the mini-socle)
is generated by the elements of prime order,
while $\Fsol (G) = G$. 

\vskip.2cm

(2)
If $G$ is the restricted sum of an infinite family $(G_n)_{n \ge 1}$ of soluble finite groups,
then $\Fsol (G)$ is the whole group $G$;
note that $\Fsol (G)$ is soluble
if and only if the supremum over all $n$ of the derived length of $G_n$ is finite.
\par

If $G$ is a finite group, then $\Fsol (G)$ is the largest soluble normal subgroup of~$G$,
known as the \textbf{soluble radical} of $G$.

\vskip.2cm

(3)
Let $G$ be a torsion-free group.
Then $\W(G) = \{e\}$, so that $\MA(G) = \MS(G) = \Fsol (G) = \{e\}$.

\vskip.2cm

(4)
Let $G$ be a group for which
Assertion \ref{2DEthm:CharP(n)} in Theorem \ref{thm:CharP(n)} holds true.
Then, with the notation of this Theorem,
the finite normal subgroup $V$ of $G$ is contained
in the abelian mini-socle $\MA (G)$.

\vskip.2cm

(5)
Let $p$ be a prime, $d$ an integer, $d \ge 2$, and $q = p^d$.
Let $C_q$ be the cyclic group $\Z / q\Z$;
denote by $c_q \in C_q$ the class modulo $q\Z$ of an integer $c \in \Z$.
Let $H_q$ be the group of triples $(a, b, c) \in \Z \times \Z \times C_q$
with the multiplication defined by
$$
(a, b, c) (a', b', c') = (a + a', b + b', c + c' + (ab')_q) .
$$
We identify the cyclic group $C_p$ of order $p$ with a subgroup of $C_q$,
and the group $C_q$ to the subgroup of $H_q$ of triples of the form $(0, 0, c)$.
Observe that all conjugacy classes in $H_q$ are finite,
i.e., $H_q$ is its own FC-centre (it is a so-called FC-group).
Moreover, the torsion FC-centre $W(H_q)$
coincides with the central subgroup $C_q$ of $H_q$,
and also with $\Fsol (H_q)$.
The following five subgroups of $H_q$
constitute a strictly ascending chain of characteristic subgroups:
\par
the trivial group $\{e\}$,
\par
the mini-socle $\MS (H_q) = \MA (H_q) = C_p$,
\par
the group $\Fsol (H_q) = \W(H_q) = C_q$, 
\par
the centre $q\Z \times q\Z \times C_q$,
\par
and the group $H_q$ itself. 

\vskip.2cm

(6)
Let $G$ be a non-trivial nilpotent group.
Since minimal normal subgroups of $G$ are central,
as recalled in Example \ref{examplenormalcentral}(5),
it follows that the mini-socle of~$G$ is the subgroup
generated by the central elements of prime order.
\par

Recall also that the set $\tau(G)$ of elements of finite order in $G$
is a subgroup of~$G$, indeed a charactersitic subgroup,
and that $G/\tau(G)$ is torsion-free.
When $G$ is moreover finitely generated then $\tau(G)$ is finite
\cite[Chapter 1, Corollary 10]{Sega--83}.
\par

It follows that, for a finitely generated nilpotent group $G$,
we have $\Fsol (G) = \W(G) = \tau(G)$,
and $\W(G / \W(G)) = \{e\}$.
The next example shows that the finite generation condition cannot be deleted.
\par

\vskip.2cm

(7)
For each integer $n \ge 1$, let 
$$
H_n = \langle x_n, y_n, z_n \mid 
x_n^3, \hskip.1cm y_n^3, \hskip.1cm
[x_n, y_n]z_n^{-1}, \hskip.1cm [x_n, z_n], \hskip.1cm [y_n, z_n] \rangle
$$
be a copy of the Heisenberg group over the field $\F_3$.
We form the full direct product $P = \prod_{n \ge 1} H_n$
and, for each $n$, we identify $x_n, y_n$, and $z_n$ with their natural images in $P$.
We also set $x = (x_n)_{n \ge 1} \in P$, and define 
$$
G = \langle x, y_n \mid n \ge 1 \rangle \le P.
$$
The group $G$ is countable, of exponent $3$, and nilpotent of class~$2$.
Observe that $z_n = [x_n, y_n]$ is in $G$ for all $n \ge 1$.
\par

For $n \ge 1$, let $A_n$ be the group generated by $y_n$ and $z_n$.
It is an abelian $3$-group of order $9$, which is normal in each of $H_n$, $P$, and $G$.
Let $A$ be the subgroup of $G$ generated by $\bigcup_{n \ge 1} A_n$,
which is normal in $G$.
Observe that $G/A$ is a cyclic group of order $3$, generated by the class of $x$ modulo $A$.
\par

We have $A \le \Fsol (G)$.
Indeed, let $t = (t_n)_{n \ge 1} \in A$.
There exists $C \ge 1$ such that $t_n = e$ whenever $n \ge C$,
so that $t$ is in the normal subgroup $\prod_{n=1}^C A_n$ of $G$, which is finite and abelian.
Hence $t \in \Fsol (G)$.
\par

For all $n \ge 1$, we have $[x, y_n] = z_n \in G$,
so that the normal subroup $\llag x \rrag_G$ generated by $x$
contains $\{z_n \mid n \ge 1\}$, and thus is infinite.
It follows that $x$ is not in the FC-centre of $G$,
and in particular that $x$ is not in $\Fsol (G)$.
\par

We have shown that $A = \Fsol (G) \lneqq G$,
and that $G / \Fsol (G)$ is a cyclic group of order $3$.
In particular, $\Fsol \big( G / \Fsol (G) \big) = G / \Fsol (G) \cong C_3$ is not trivial.
\par

Since $\W (-)$ and $\Fsol (-)$ coincide for $G$ and its quotients
(indeed for any soluble group), this can be written
$A = \W(G) \lneqq G$ and $\W(G/\W(G)) = G/\W(G) \cong C_3$.
\end{exe}

The last example shows that $\Fsol (G)$
does not behave as a \textit{radical} in general,
in the sense that $\Fsol (G / \Fsol (G))$ can be non-trivial.
Similarly $W(G / \W(G))$ can be non-trivial.
However, it is easy to see that, if $\W(G)$ is finite [respectively $\Fsol (G)$ is finite],
then $\W(G / \W(G)) =\{e\}$ [respectively $\Fsol (G / \Fsol (G)) = \{e\}$].

The following proposition will be used in Remark \ref{TrunkSufFiur}(2).

\begin{prop}
\label{MSMATofproduct}
For any two groups $G_1$ and $G_2$, we have:
\begin{enumerate}[label=(\roman*)]
\item\label{iDEMSMATofproduct}
$\MS(G_1 \times G_2) = \MS(G_1) \times \MS(G_2)$,
\item\label{iiDEMSMATofproduct}
$\MA(G_1 \times G_2) = \MA(G_1) \times \MA(G_2)$,
\item\label{iiiDEMSMATofproduct}
$ \Fsol (G_1 \times G_2) = \Fsol (G_1) \times \Fsol (G_2)$.
\item\label{ivDEMSMATofproduct}
$ \W(G_1 \times G_2) = \W(G_1) \times \W (G_2)$.
\end{enumerate}
\end{prop}

\begin{proof}
We identify $G_1$ and its subgroups with subgroups of $G_1 \times G_2$,
and similarly for $G_2$ and its subgroups.
For $j \in \{1, 2\}$, we denote by $e_j$ the neutral element of $G_j$
and by $r_j \hskip.1cm \colon G_1 \times G_2 \twoheadrightarrow G_j$
the canonical projection.

\vskip.2cm

\ref{iDEMSMATofproduct}
The inclusion $\MS(G_1) \times \MS(G_2) \le \MS(G_1 \times G_2)$ is straightforward,
because any minimal non-trivial finite normal subgroup of $G_1$ or of $G_2$
is a minimal non-trivial finite normal subgroup of $G_1 \times G_2$.

\par

To check the reverse inclusion,
consider a minimal non-trivial finite normal subgroup $N$ of $G_1 \times G_2$,
and distinguish two cases.
First, if $N \le G_1$ or $N \le G_2$, then $N \le \MS(G_1) \times \MS(G_2)$.
Second, if $N \nleq G_1$ and $N \nleq G_2$,
then $N$ does not contain any element of the form $(x_1, e_2)$ or $(e_1, x_2)$
with $x_1 \ne e_1$ and $x_2 \ne e_2$, by minimality.
If $N$ did contain an element $x = (x_1, x_2)$ with $x_1$ non central in $G_1$,
then $N$ would contain $(y_1, e_2)^{-1} x^{-1} (y_1, e_2) x = ([y_1, x_1], e_2)$
for some $y_1 \in G_1$ such that $[y_1, x_1] \ne e_1$,
in contradiction with the hypothesis on $N$;
and similarly for $N \ni (x_1, x_2)$ with $x_2$ non central in $G_2$; 
hence $r_1(N)$ is central in $G_1$ and $r_2(N)$ is central in $G_2$.
It follows that $N$ is central in $G_1 \times G_2$,
and that there exists a prime $p$ such that $N$ is a cyclic group of order~$p$.
Hence $N$ is of the form $\langle (x', x'') \rangle_{G_1 \times G_2}$
with $x'$ of order $p$ in $G_1$ and $x''$ of order $p$ in $G_2$.
In particular,
$N \le \llag x' \rrag_{G_1} \times \llag x'' \rrag_{G_2} \le \MS (G_1) \times \MS (G_2)$.
It follows that $\MS (G_1 \times G_2) \le \MS (G_2) \times \MS (G_2)$.

\vskip.2cm

An argument of the same kind shows that \ref{iiDEMSMATofproduct} holds.

\vskip.2cm

\ref{ivDEMSMATofproduct}
Given $x \in \W(G_1 \times G_2)$,
the normal closure $\llag x \rrag_{G_1 \times G_2}$ is finite by definition.
Therefore $r_j(\llag x \rrag_{G_1 \times G_2}) = \llag r_j(x) \rrag_{G_j}$ is finite,
so that $r_j(x) \in \W(G_j)$, for $j = 1, 2$.
This proves that $x \in \W(G_1) \times \W(G_2)$,
hence $\W(G_1 \times G_2) \le \W(G_1) \times \W(G_2)$. 
\par

Let $j \in \{1,2\}$ and $x_j \in G_j$.
The group $G_{3-j}$ commutes with $x_j$,
so that $\llag x_j \rrag_{G_1 \times G_2} = \llag x_j \rrag_{G_j}$.
Assume in addition that $x_j \in \W(G_j)$;
then by definition $\llag x_j \rrag_{G_j}$ is finite,
hence $\llag x_j \rrag_{G_1 \times G_2}$ is finite as well. 
Therefore $x_j \in \W(G_1 \times G_2)$.
This proves that $\W(G_j) \le \W(G_1 \times G_2)$.
Therefore $\W(G_1) \times \W(G_2) \le \W(G_1 \times G_2)$, which ends the proof of \ref{ivDEMSMATofproduct}.

\vskip.2cm

An argument of the same kind shows that \ref{iiiDEMSMATofproduct} holds.
\end{proof}

\subsection{A basic property of factor representations}
\label{subsec:factor}

\begin{lem}
\label{lem:Disjoint}
Let $\pi$ be a unitary representation of a group $G$ in a Hilbert space $\mathcal H$
and $N$ a normal subgroup of $G$.
Let $\pi_1$ [respectively $\pi_2$] be the subrepresentation of $\pi$
given by the $G$-action on the subspace $\mathcal H^N$ of $\mathcal H$
consisting of the $N$-invariant vectors [respectively on its orthogonal complement].
\par

Then $\pi_1$ and $\pi_2$ are disjoint. 
\end{lem}

\begin{proof}
Let $\rho_1$ be a non-zero subrepresentation of $\pi_1$
and $\rho_2$ a non-zero subrepresentation of $\pi_2$.
On the one hand, the kernel of $\rho_1$ contains $N$.
On the other hand, if the kernel of $\rho_2$ did contain $N$,
the space of $\rho_2$ would be contained in $\mathcal H^N$,
hence it would be $\{0\}$ by the definition of $\pi_2$.
This is preposterous. Therefore the representations $\rho_1$ and $\rho_2$
have different kernels, and thus they are not equivalent.
\end{proof}

Two unitary representations $\pi$, $\pi'$ of a group $G$ are called \textbf{quasi-equivalent}
if no non-zero subrepresentation of $\pi$ is disjoint from $\pi'$, and vice-versa.
	
\begin{prop}
\label{prop:KernelPrimary}
Let $\pi$ be a factor representation of a group $G$.
\par

For every non-zero subrepresentation $\rho \le \pi$, we have $\Ker(\rho) = \Ker(\pi)$.
\par

In particular, if $\pi'$ is any factor representation quasi-equivalent to $\pi$,
then $\Ker(\pi) = \Ker(\pi')$.
\end{prop}

\begin{proof}
Set $N = \Ker(\rho)$. 
Denote by $\mathcal H_\pi$ the Hilbert space of $\pi$
and by $\mathcal H_\rho$ that of~$\rho$.
\par

Since $\rho \le \pi$, we have $\Ker(\pi) \le N$.
When $N = \{e\}$, there is nothing more to prove.
We assume now that $N \ne \{e\}$.
\par

The space $\mathcal H_\rho$ is contained in $\mathcal H_\pi^N$;
in particular $\mathcal H_\pi^N \ne \{0\}$ since $\rho$ is non-zero.
Since $\pi$ is a factor representation,
$\mathcal H_\pi^N = \mathcal H_\pi$ by Lemma~\ref{lem:Disjoint}.
It follows that $N \le \Ker(\pi)$, hence that $N = \Ker(\pi)$.
\par

Let $\pi'$ be a factor representation of $G$ which is quasi-equivalent to $\pi$.
By \cite[Theorem 1.7, Page 20]{Mack--76},
up to equivalence we have $\pi \le \pi'$ or $\pi' \le \pi$.
Hence $\Ker(\pi) = \Ker(\pi')$ by the assertion that we have already established.
\end{proof}

\subsection{On \texorpdfstring{$G$}{G}-faithful representations
of subgroups of \texorpdfstring{$G$}{G}}
\label{subsec:SubsidiaryBeHa}

Given a group $G$ and a normal subgroup $N$,
a unitary character or a {unitary representation
$\rho$ of~$N$ is called \textbf{$G$-faithful}
if the intersection over all $g \in G$ of the kernels $\Ker (\rho^g)$ is trivial,
where $\rho^g (x) = \rho(gxg^{-1})$ for all $x \in N$. 
\par

The following lemma generalizes \cite[Lemma~9]{BeHa--08}.
More precisely, the statement from loc.\ cit.\ assumes that $\pi$ is irreducible and faithful,
whereas we only require that $\pi$ is a factor representation
and that the restriction $\pi \vert_N$ is faithful. 

\begin{lem}
\label{FromL9}
Let $G$ be a countable group, $N$ a normal subgroup of $G$,
and $\pi$ a factor representation of $G$
such that the restriction $\pi \vert_N$ is faithful.
\par

Then $N$ has an irreducible unitary representation $\rho$ which is $G$-faithful.
\end{lem}

\begin{proof}
The proof is a small modification of that in \cite[Lemma~9]{BeHa--08}.
We reproduce the details since our hypotheses are slightly more general.
\par

We assume that $N$ is non-trivial, since otherwise there is nothing to prove. 
Set $\sigma := \pi \vert_N$
and let $\sigma = \int_\Omega^\oplus \sigma_\omega d\mu(\omega)$
be a direct integral decomposition of $\sigma$ into irreducible unitary representations,
implemented by an isomorphism
$\mathcal H_\sigma \cong \int_\Omega^\oplus \mathcal H_\omega d\mu(\omega)$.
Denote by $\{C_j\}_{j \in J}$ the family of $G$-conjugacy classes
contained in $N$ distinct from $\{e\}$.
For each $j$, let $N_j \le N$ be the subgroup generated by $C_j$;
note that $N_j$ is normal in $G$.
The family $\{C_j\}_{j \in J}$ is countable and non-empty.
Every non-trivial normal subgroup of $G$ contained in $N$
must contain $N_j$ for some $j \in J$.
Therefore, given $\omega \in \Omega$, we see that
$\sigma_\omega$ is not $G$-faithful if and only if 
$\Ker \big( \bigoplus_{g \in G} (\sigma_\omega)^g \big)$ contains $N_j$ for some $j \in J$. 
\par

Set now
$\Omega_j = \big\{ \omega \in \Omega \mid
N_j \le \Ker \big( \bigoplus_{g \in G} (\sigma_\omega)^g \big) \big\}$
and $\widetilde \Omega = \bigcup_{j \in J} \Omega_j$.
It follows that $\widetilde \Omega$ is the subset consisting of these $\omega \in \Omega$
such that $\sigma_\omega$ is not $G$-faithful.
By \cite[Lemma~8]{BeHa--08}, each $\Omega_j$ is measurable.
Since $J$ is countable, $\widetilde \Omega$ is also measurable.
\par

In order to finish the proof, it suffices to show that $\mu(\widetilde \Omega) = 0$.
Suppose for a contradiction that $\mu(\widetilde \Omega)>0$.
Since $J$ is countable, we have $\mu(\Omega_\ell)>0$ for some $\ell \in J$.
For each $\omega \in \Omega_\ell$, we have $N_\ell \le \Ker(\sigma_\omega)$,
so that the subspace $\int_{\Omega_\ell}^\oplus \mathcal H_\omega$ of $\mathcal H_\sigma$,
which is non-zero since $\mu(\Omega_\ell)>0$, consists of $N_\ell$-invariant vectors.
Since $N_\ell$ is normal in $G$, the set of $N_\ell$-invariant vectors is $G$-invariant,
and thus corresponds to a subrepresentation of $\pi$.
Since $\pi$ is a factor representation,
we have $N_\ell \le \Ker(\pi)$ by Proposition~\ref{prop:KernelPrimary}.
Since $N_\ell \le N$, this contradicts the hypothesis that $\pi \vert_N$ is faithful. 
\par

We have just shown that {almost all} irreducible unitary representations $\sigma_\omega$ of $N$
occurring in a direct integral decomposition of $\sigma$ are $G$-faithful.
In particular there exists $\omega \in \Omega$
such that the irreducible unitary representation $\rho := \sigma_\omega$ is $G$-faithful.
\end{proof}

\begin{proof}[Proof of Proposition~\ref{prop:Factor->Irrep}]
Let $\pi$ be a factor representation of the countable group $G$.
View $\pi$ as a faithful representation of the group $H := G / \Ker(\pi)$.
By Lemma~\ref{FromL9} applied to $H$ and its trivial normal subgroup $N = H$,
the group $H$ has an irreducible unitary representation $\rho$ which is faithful.
We may now view $\rho$ as a representation of $G$
and the proposition follows. 
\end{proof}

\begin{lem}
\label{lem:G-faithful-normal-subgroup}
Let $G$ be a countable group, $N$ a normal subgroup of $G$,
and $\sigma$ an irreducible unitary representation of $N$
which is $G$-faithful.
\par

Then $G$ has an irreducible unitary representation $\pi$ with the following properties:
the restriction $\pi \vert_N$ is faithful,
and every element of $\Ker(\pi)$ is contained in a finite normal subgroup of $G$,
i.e., $\Ker(\pi)$ is contained in the torsion FC-centre $\W(G)$ of $G$.
\end{lem}

\begin{proof}
Let $\rho = \mathrm{Ind}_N^G(\sigma)$ be the unitary representation of $G$
induced from $\sigma$.
Let $\rho = \int_\Omega^\oplus \rho_\omega d\mu(\omega)$
be a direct integral decomposition of $\rho$ into irreducible unitary representations. 
Set
$$
\widetilde \Omega = \{\omega \in \Omega
\mid
\rho_\omega \vert_N \text{ is not faithful} \}
$$
and
$$
\widehat \Omega = \{\omega \in \Omega
\mid
\text{there exists} \hskip.2cm
g \in \Ker(\rho_\omega)
\hskip.2cm \text{such that} \hskip.2cm
\llag g \rrag_G \hskip.2cm \text{is infinite} \}.
$$
We claim that $\mu(\widetilde \Omega) = \mu( \widehat \Omega) = 0$;
to show this, we argue as in the proof of \cite[Lemma~10]{BeHa--08}.
\par

To show that $\mu(\widetilde \Omega) = 0$, 
we proceed by contradiction.
We assume that there exists
a conjugacy class $C_\ell \ne \{e\}$ of $G$ contained in $N$,
generating a subgroup $G_\ell$ of $G$ which is normal and contained in $N$,
and defining a measurable subset
$\Omega_\ell = \{ \omega \in \Omega \mid G_\ell \le \Ker (\rho_\omega) \}$,
such that $\mu(\Omega_\ell) > 0$.
Then, as in `Claim 1' in the proof of \cite[Lemma 10]{BeHa--08}
we show that $G_\ell \cap N = \{e\}$, in contradiction with $G_\ell \le N$.
\par

To show that $\mu(\widehat \Omega) = 0$,
also by contradiction, 
we assume this time that there exists
a conjugacy class $C_m \ne \{e\}$ of $G$
generating an infinite subgroup $G_m$ of $G$,
and defining a measurable subset
$\Omega_m = \{ \omega \in \Omega \mid G_m \le \Ker (\rho_\omega) \}$,
such that $\mu(\Omega_m) > 0$,
and we arrive at a contradiction.
Indeed, `Claim 1' in the proof already quoted shows that $G_m \cap N = \{e\}$,
and `Claim 2' in the same proof shows that $G_m$ is finite, 
in contradiction with the hypothesis.
\par

Consequently, the complement of $\widetilde \Omega \cup \widehat \Omega$ in $\Omega$
has full measure, and thus is non-empty.
For any $\omega \in \Omega \smallsetminus (\widetilde \Omega \cup \widehat \Omega)$,
the representation $\pi := \rho_\omega$ is an irreducible unitary representation of $G$
that has the required properties.
\end{proof}

A strengthening of Lemma~\ref{lem:G-faithful-normal-subgroup}
will be established in Lemma~\ref{lem:G-faithful-normal-subgroup:Kernel<sol(G)} below.
 
\begin{lem}
\label{lem:semi-simple_feet_G-faithful}
Let $G$ be a group and $N, A, S$ normal subgroups of $G$
such that $N = A \times S$.
Assume that $A$ is abelian,
and that $S$ is the restricted sum
of a collection $\{S_i\}$ of non-abelian simple finite groups. Then:
\begin{enumerate}[label=(\roman*)]
\item
$S$ has a faithful irreducible unitary representation;
\item
$N$ has a $G$-faithful irreducible unitary representation if and only if
$A$ has a $G$-faithful unitary character.
\end{enumerate}
\end{lem}

\begin{proof}[Proof:]
see Lemma 13 and its proof in \cite{BeHa--08}. 
\end{proof}

We end this section with some subsidiary facts.
Given an abelian group $A$,
denote by $\widehat A$ the \textbf{Pontrjagin dual} of $A$,
namely the space of all unitary characters
$A \to \T$,
with the compact open topology.
Recall that $\widehat A$ is a compact abelian group.

\begin{lem}
\label{lem:Pontrjagin}
Let $G$ be a discrete group, $A$ an abelian normal subgroup of $G$,
and $\chi$ a unitary character of $A$.
\par

Then $\chi$ is $G$-faithful if and only if
the subgroup generated by 
$\chi^G = \{\chi^g \mid g \in G\}$ is dense in $\widehat A$.
\end{lem}

\begin{proof}
This follows from Pontrjagin duality.
See the proof of the equivalence between (i) and (ii) of Lemma 14 in in \cite{BeHa--08}.
\end{proof}

Before the last proposition of this section, we recall
the natural module structure on abelian normal subgroups,
the definition of cyclic modules,
and we state a lemma which is helpful for translating
from the language of abelian groups to that of modules.

\begin{rem}
\label{remabnormalmodule}
Let $G$ be a group, $V$ an abelian normal subgroup of $G$,
and $\Z[G]$ the group ring of $G$ over the integers.
Then $V$ has a canonical structure of $\Z[G]$-module.
Moreover, $V$ is simple as a $\Z[G]$-module if and only if
$V$ is minimal as abelian normal subgroup of $G$.
\par

Compare with the reminder on simple $\F_p[G]$-modules
just before Theorem \ref{thm:CharP(n)}.
\end{rem}

For a ring $R$ and a module $V$,
the module~$V$ is \textbf{cyclic} if there exists $v \in V$ such that $V = Rv$.
This terminology is used below
for $R$ the group ring $\Z[G]$ and $V$ an abelian normal subgroup of $G$,
and for $R$ the group algebran $\F_pG]$
and $V$ a $p$-elementary abelian normal subgroup of $G$ for some prime $p$.
\par

The proof of the next lemma is straightforward,
and left to the reader.

\begin{lem}
\label{lem:cyclic1conjclass}
Let $G$ be a group and $V$ an abelian normal subgroup of $G$.
\par

Then $V$ is generated as a group by one $G$-conjugacy class
if and only if $V$ as a $\Z[G]$-module is cyclic.
\par

Suppose moreover that $V$ is an elementary abelian $p$-group.
Then $V$ is generated as a group by one $G$-conjugacy class
if and only if $V$ as an $\F_p[G]$-module is cyclic.
\end{lem}

The following classical result will be frequently used in the sequel,
without further notice.
For a proof, see \cite[\S~3, no.\ 3]{Bo--A8}.

\begin{prop}
\label{prop:semi-simple}
Let $R$ be a ring and $U$ a $R$-module.
The following conditions are equivalent:
\begin{enumerate}[label=(\roman*)]
\item\label{iDEprop:semi-simple}
$U$ is generated by simple submodules.
\item\label{iiDEprop:semi-simple}
$U$ is a direct sum of a family of simple submodules.
\item\label{iiiDEprop:semi-simple}
Every submodule of $U$ is a direct summand. 
\end{enumerate}
If $U$ satisfies these conditions, then
\begin{enumerate}[label=(\alph*)]
\item
every submodule of $U$ satisfies
Conditions \ref{iDEprop:semi-simple} to \ref{iiiDEprop:semi-simple},
\item
every quotient module of $U$ satisfies
Conditions \ref{iDEprop:semi-simple} to \ref{iiiDEprop:semi-simple}.
\end{enumerate}
\end{prop}

A module $U$ satisfying Conditions \ref{iDEprop:semi-simple} to \ref{iiiDEprop:semi-simple}
is called \textbf{semi-simple}. 

\vskip.2cm

Proposition \ref{lem14BEHA} will be needed in Section~\ref{Section:proofof1.1}.

\begin{prop}
\label{lem14BEHA}
Let $G$ be a group,
$A$ a finite normal subgroup of $G$ contained in $\MA(G)$,
and $p$ a prime.
\par

The following properties are equivalent:
\begin{enumerate}[label=(\roman*)]
\item \label{iDElem14BEHA}
The group $A$ has a $G$-faithful unitary character.
\item \label{iiDElem14BEHA}
The group $A$ is generated by a single conjugacy class.
\item \label{iiiDElem14BEHA}
The $\Z[G]$-module $A$ is cyclic.
\end{enumerate}
Suppose moreover that $A$ is an elementary abelian $p$-group.
Then Properties \ref{iDElem14BEHA} to \ref{iiiDElem14BEHA} are equivalent to:
\begin{enumerate}[label=(\roman*)]
\addtocounter{enumi}{3}
\item \label{ivDElem14BEHA}
The $\F_p[G]$-module $A$ is cyclic.
\end{enumerate}
\end{prop}

\begin{proof}
For the equivalence of \ref{iDElem14BEHA} and \ref{iiDElem14BEHA},
we follow the arguments of the proof of	Lemma~14 in \cite{BeHa--08}
(whose formal statement is however insufficient for our purposes). 
\par 

By \ref{2DEstructureminisocle} in Proposition \ref{structureminisocle},
$A$ is a finite abelian group and is therefore
a direct sum $A = \bigoplus_{p \in P} A_p$,
where $P$ is the set of primes $p$ for which $A$ has elements of order $p$,
and where $A_p$ is the $p$-Sylow subgroup of $A$.
Moreover $A_p$ is an elementary abelian $p$-group for each $p \in P$,
by (1) of the same proposition. 
Notice that $A_p$ is semi-simple by Proposition~\ref{prop:semi-simple},
since $A$ is contained in $\MA(G)$.
 (For comparison with \cite[Lemma~14]{BeHa--08}, note that
it follows from Proposition \ref{prop:semi-simple} applied to each $A_p$
that there exists a finite set $\{A_i\}_{i \in E}$ of abelian mini-feet in $G$
such that $A = \bigoplus_{i \in I} A_i$;
each $A_i$ is isomorphic to $(\F_p)^n$ for some prime $p$ and some $n \ge 1$.)
Observe that the Pontryagin dual of $A = \bigoplus_{p \in P} A_p$
is canonically isomorphic to $\bigoplus_{p \in P} \widehat A_p$.
\par

We know by Lemma~\ref{lem:Pontrjagin} that
$A$ has a $G$-faithful unitary character if and only if
$\widehat A$ is generated by one $G$-orbit.
By the Chinese Remainder Theorem,
the group $\widehat A = \bigoplus_{p \in P} \widehat A_p$
is generated by a single $G$-orbit
if and only each of its $p$-Sylow subgroups $\widehat A_p$
is generated by a single $G$-orbit
(this can alternatively be deduced from Lemma~\ref{lem:cyclic1conjclass}
together with Lemma~\ref{lem:IsotypicalDecomposition} below).
Using Lemma~\ref{lem:Pontrjagin} again,
we deduce that $A$ has a $G$-faithful unitary character
if and only if $A_p$ has a $G$-faithful character for each $p \in P$.
\par

Consequently, it suffices to prove 
the equivalence of \ref{iDElem14BEHA} and \ref{iiDElem14BEHA}
when $A = A_p$ for one prime $p$.
By Lemma~\ref{lem:cyclic1conjclass},
the group $A_p$ is generated by a single conjugacy class
if and only if $A_p$ is cyclic as an $\F_p[G]$-module. 
Under the natural identification of $\widehat A_p$
with the dual $A_p^* := \text{Hom}_{\F_p}(A_p, \F_p)$,
the $G$-action on $\widehat A_p$ corresponds to
the dual (or contragredient) action of $G$ on $A_p^*$.
Thus we may identify $\widehat A_p$ with $A_p^*$ as $\F_p[G]$-modules.
A finite semi-simple $\F_p[G]$-module is cyclic
if and only if its dual is cyclic (see Lemma 3.2 in \cite{Szec--16}).
Since the dual $A_p^*$ is canonically isomorphic to the Pontrjagin dual $\widehat A_p$,
and since $A_p$ is semi-simple,
we deduce from Lemma~\ref{lem:Pontrjagin}
that $A_p$ is generated by a single conjugacy class
if and only if $A_p$ has a $G$-faithful unitary character. 

\vskip.2cm

The equivalence of \ref{iiDElem14BEHA} and \ref{iiiDElem14BEHA}
holds by Lemma \ref{lem:cyclic1conjclass}.

\vskip.2cm

In the particular case of $A$ an elementary abelian $p$-group, similarly,
the equivalence of \ref{iiDElem14BEHA} and \ref{ivDElem14BEHA}
holds by Lemma \ref{lem:cyclic1conjclass}.
\end{proof}

\subsection{Irreducible representations whose kernel
is contained in \texorpdfstring{$\Fsol (G)$}{FSol (G)}}

The goal of this subsection is to establish the following result of independent interest. 

\begin{prop}
\label{prop:Kernel}
Any countable group $G$
admits an irreducible unitary representation $\pi$ such that,
for every element $g \in \Ker(\pi)$,
the normal closure $\llag g \rrag_G$ is a soluble finite subgroup of $G$.
\par

In other words, $G$ has an irreducible unitary representation
whose kernel is contained in the characteristic subgroup $\Fsol (G)$.
\end{prop}

We need the following. 

\begin{lem}
\label{lem:ShrinkKernel}
Let $G$ be a countable group and $K$ a normal subgroup of $G$
contained in the torsion FC-centre $\W(G)$. 
\par 

If $G / K$ is irreducibly faithful, then $G / (K \cap \Fsol (G))$ is also irreducibly faithful. 
\end{lem}

\begin{proof}
Set $S = \Fsol (G)$. 
In order to show that $G / (K \cap S)$ is irreducibly faithful,
it suffices by Theorem~\ref{thm:Gasch} to consider
an arbitrary finite normal subgroup $A$ of $G / (K \cap S)$
contained in $\MA \big(G / (K \cap S) \big)$
and to show that $A$ is generated by a single conjugacy class.
\par

Let $r_1 \hskip.1cm \colon G \twoheadrightarrow G / (K \cap S)$
and $r_2 \hskip.1cm \colon G / (K \cap S) \twoheadrightarrow G / K$
be the canonical projections.
We claim that the restriction $r_2 \vert_A$ is injective.
Indeed, let $x \in G$ be such that $r_1(x) \in \Ker(r_2\vert_A) = A \cap \Ker(r_2)$;
note that $r_1(x) \in A$ and $x \in K$.
We have
$$
\llag r_1(x) \rrag_{G / (K \cap S)} \cong
\llag x \rrag_G / \big( \llag x \rrag_G \cap (K \cap S) \big) =
\llag x \rrag_G / \big( \llag x \rrag_G \cap S \big) .
$$
Since $K \le \W(G)$ by hypothesis,
the normal closure $\llag x \rrag_G$ is finite.
By the definition of $S$, every finite normal subgroup of $G$ contained in $S$ is soluble;
hence $\llag x \rrag_G \cap S$ is soluble.
Moreover $\llag r_1(x) \rrag_{G / K \cap S}$ is abelian
beause $r_1(x) \in A$ and $A$ is abelian normal in $G / (K \cap S)$.
It follows that $\llag x \rrag_G$ is soluble-by-abelian, hence soluble.
We infer that $x \in S$. Therefore $r_1(x) = e$, which proves the claim.
\par

Since $G / (K \cap S)$ is a quotient of $G$, we may view $A$ as a $\Z[G]$-module,
and we must show that this module is cyclic (see Proposition~\ref{lem14BEHA}).
The claim implies that $r_2$ induces an isomorphism of $\Z[G]$-modules $A \to r_2(A)$.
Since $G / K$ is irreducibly faithful by hypothesis,
and since $r_2(A) \le \MA(G / K)$ by Proposition~\ref{structureminisocle}\ref{7DEstructureminisocle},
we deduce from Theorem~\ref{thm:Gasch}
that $r_2(A)$ is generated by a single conjugacy class in $G / K$.
Thus $r_2(A)$ is a cyclic $\Z[G]$-module by Proposition~\ref{lem14BEHA},
from which it finally follows that $A$ is a cyclic $\Z[G]$-module, as required. 
\end{proof}

\begin{proof}[Proof of Proposition~\ref{prop:Kernel}]
The group $G$ has an irreducible unitary representation
whose kernel $K$ is contained in $\W(G)$,
by Lemma~\ref{lem:G-faithful-normal-subgroup} applied with $N = \{e\}$.
By Lemma~\ref{lem:ShrinkKernel},
it follows that $G$ also has an irreducible unitary representation
whose kernel is contained in $\Fsol (G)$. 
\end{proof}

\begin{rem}
\label{BrolineGarrison}
For a finite group $G$,
Proposition \ref{prop:Kernel} implies that $G$
has an irreducible representation with soluble kernel.
This falls quite short of a theorem
due to Broline and Garrison \cite[Corollary 12.20]{Isaa--76}
which establishes that $G$ has an irreducible representation with nilpotent kernel.
More precisely:
\par

\emph{
Let $G$ be a finite group and let $\pi$ be an irreducible representation of $G$ over $\mathbf C$
satisfying either of the following conditions:
(i) the degree of $\pi$ is maximal among the degrees of all irreducible representations of $G$,
(ii) the kernel of $\pi$ is minimal among the kernels of all irreducible representations of $G$.
Then the kernel of $\pi$ is nilpotent.
}
\par

There are groups without any irreducible representation having abelian kernel.
This is well-known to experts, and we are convinced that examples exist in the literature,
but we have not been able to find a precise reference;
one specific example can be found in Appendix \ref{Section:appendixA}.
\end{rem}

\begin{rem}
\label{TrunkSufFiur}
Let $G$ be a countable group.

\vskip.2cm

(1)
It follows from Proposition \ref{prop:Kernel} that the
complement $G \smallsetminus \Fsol (G)$ of $\Fsol (G)$ 
in a countable group $G$ is irreducibly faithful.
(A refinement of that statement will be established
in Proposition~\ref{prop:IrredFaithful:sol}.)

\vskip.2cm

(2)
However, the quotient $G / \Fsol (G)$
need not have a faithful irreducible unitary representation.
\par

Indeed, let $H$ be a countable group
such that $H / \Fsol (H) \cong C_3$ is cyclic of order $3$;
see Example \ref{exe:socletrunk}(7).
Set $G = H \times H$.
By Proposition~\ref{MSMATofproduct}, we have
$$
G / \Fsol (G) \cong (H \times H) / (\Fsol (H) \times \Fsol (H)) \cong C_3 \times C_3 ,
$$
so that $G / \Fsol (G)$ does not have any faithful irreducible unitary representation.

\vskip.2cm

(3)
In \cite[Corollary 3]{BeHa--08}, it is noted that
each of the following conditions on $G$ is sufficient
to imply that $G$ has a faithful irreducible unitary representation:
\begin{enumerate}[label=(\roman*)]
\item\label{iDETrunkSufFiur}
$G$ is torsion-free,
\item\label{iiDETrunkSufFiur}
all conjugacy classes in $G$ distinct from $\{e\}$ are infinite.
\end{enumerate}
Proposition \ref{prop:Kernel} shows that the following condition,
weaker than both \ref{iDETrunkSufFiur} and \ref{iiDETrunkSufFiur},
is also sufficient:
\begin{enumerate}[label=(\roman*)]
\addtocounter{enumi}{2}
\item\label{iiiDETrunkSufFiur}
$\Fsol (G) = \{e\}$.
\end{enumerate}
Here are two families of groups for which \ref{iiiDETrunkSufFiur} holds,
but neither \ref{iDETrunkSufFiur} nor \ref{iiDETrunkSufFiur} does.
\par

(a)
Any restricted sum $H$ of finite non-abelian simple groups.
More generally,
any direct product $G \times H$ of an irreducibly faithful group $G$
with a restricted sum $H$ of finite non-abelian simple groups.
\par

(b)
Let $G$ be one of the groups defined by B.H.\ Neumann in 1937
to show that there are uncoutably many pairwise non-isomorphic groups
which are finitely generated and not finitely presented;
see \cite{Neum--37}, as well as 
\cite[Section III.B and in particular no~35]{Harp--00}.
Recall that, in $G$, the FC-centre is a restricted sum
$N := \prod'_n H_n$, where each $H_n$ is a finite simple alternating group,
and $\vert H_1 \vert < \dots < \vert H_n \vert < \vert H_{n+1} \vert < \dots$,
and the quotient $G/N$ is the permutation group of $\Z$
generated by translations and even finitely supported permutations.
Observe that $G$ is neither torsion-free,
nor with all conjugacy classes other than $\{e\}$ infinite.
The subgroup $\Fsol (G)$ is trivial, and therefore $G$ has
a faithful irreducible unitary representation.
\end{rem}

We finish by recording the following strengthening
of Lemma~\ref{lem:G-faithful-normal-subgroup},
which will be needed in Section~\ref{Section:proofof1.1}. 

\begin{lem}
\label{lem:G-faithful-normal-subgroup:Kernel<sol(G)}
Let $G$ be a countable group, $N$ a normal subgroup of $G$,
and $\sigma$ an irreducible unitary representation of $N$
which is $G$-faithful.
\par

Then $G$ has an irreducible unitary representation $\pi$
such that $\Ker(\pi) \cap N =\{e\}$ and $\Ker(\pi) \le \Fsol (G)$.
\end{lem}

\begin{proof}
Let $K$ be the kernel of the irreducible unitary representation of $G$
afforded by applying Lemma~\ref{lem:G-faithful-normal-subgroup} to $\sigma$.
Thus $K \cap N = \{e\}$ and $K \le \W(G)$.
The desired conclusion now follows from Lemma~\ref{lem:ShrinkKernel}.
\end{proof}

\section{Cyclic semi-simple \texorpdfstring{$\F_p[G]$}{Fp[G]}-modules}
\label{Section:cyclic}

Let $R$ be a ring.
The following lemma is the module version
of a result often stated for groups and known as Goursat's Lemma. 
The module version appears, for example, in \cite[Page 171]{Lamb--76}.
More on this lemma can be consulted in \cite{BaSZ--15}.

\begin{lem}
\label{lem:Goursat}
Let $A = A_0 \oplus A_1$ be the direct sum of two $R$-modules.
For $i = 0, 1$, let $r_i \hskip.1cm \colon A \twoheadrightarrow A_i$
be the canonical projection.
Let $M \le A$ be a submodule such that $r_i(M) = A_i$ for $i = 0, 1$.
Set $M_i = M \cap A_i$. 
\par

Then the $R$-modules $A_0/M_0$ and $A_1/M_1$ are isomorphic,
and the the canonical image of $M$ in $A_0/M_0 \oplus A_1/M_1$
is the graph of an isomorphism of $R$-modules $A_0/M_0 \to A_1/M_1$. 
\end{lem}

Let now $p$ be a prime and $G$ a group.
The goal of this section is to characterize
when a finite semi-simple $\F_p[G]$-module is cyclic.
This will be achieved in Proposition~\ref{prop:cyclicsemi-simple} below,
after some preparatory steps. 
Proposition~\ref{prop:cyclicsemi-simple} is well-known to experts:
see Lemma 3.1 in \cite{Szec--16}.
It can be seen as a version over $\F_p$
of a result for cyclic unitary representations of compact groups
which appears in Greenleaf and Moskowitz \cite[Proposition 1.8]{GrMo--71}.

\begin{lem}
\label{lem:Goursat-for-2}
Let $W$ be a finite simple $\F_p[G]$-module.
Let $\ko = \mathcal L_{\F_p[G]}(W)$ be its centralizer,
which is a finite field extension of $\F_p$.
Let $V_0, V_1$ be two copies of $W$. 
\par

Every simple $\F_p[G]$-submodule $M$ of $V_0 \oplus V_1$
such that $M \cap V_0 = \{0\}$ is of the form 
$$
M = \{(\lambda x, x) \mid x \in V_1\}
$$
for some $\lambda \in \ko$.
\end{lem}

\begin{proof}
This is a straightforward consequence of Lemma~\ref{lem:Goursat}.
\end{proof}

The following extension to a direct sum of $\ell+1$ components will be useful. 

\begin{lem}
\label{lem:Goursat-for-ell+1}
Let $W$ be a finite simple $\F_p[G]$-module.
Let $\ko = \mathcal L_{\F_p[G]}(W)$.
Let $\ell \ge 0$; for each $i = 0, \dots, \ell$, let $V_i$ be a copy of $W$.
Set $U = V_0 \oplus V_1 \oplus \dots \oplus V_\ell$. 
\par

Every maximal $\F_p[G]$-submodule $M \lneqq U$
such that $M \cap V_0 = \{0\}$
is of the form 
$$
M = \Big\{ \Big( \sum_{i=1}^\ell \lambda_i x_i, x_1, x_2, \dots, x_\ell \Big)
\hskip.2cm \Big\vert \hskip.2cm
(x_1, \dots, x_\ell) \in V_1 \oplus \dots \oplus V_\ell \Big\}
$$
for some $(\lambda_1, \dots, \lambda_\ell) \in \ko^\ell$. 
\end{lem}

\begin{proof}
Let $r \hskip.1cm \colon U \twoheadrightarrow V_1 \oplus \dots \oplus V_\ell$
be the canonical projection.
Let $M \lneqq U$ be a maximal $\F_p[G]$-submodule
such that $M \cap V_0 = \{0\}$. Then the restriction $r \vert_M$ is injective.
Since $M$ is maximal, we have $U = V_0 \oplus M$,
so that $r \vert_M \hskip.1cm \colon M \to V_1 \oplus \dots \oplus V_\ell$
is an isomorphism of $\F_p[G]$-modules. 
\par

Given $i \in \{1, \dots, \ell\}$, let $M_i = (r \vert_M)^{-1}(V_i)$.
Then $M_i$ is isomorphic to $V_i$,
hence it is a simple $\F_p[G]$-submodule of $M$ contained in $V_0 \oplus V_i$.
Moreover $M_i \cap V_0 = \{0\}$. By Lemma~\ref{lem:Goursat-for-2},
there exists $\lambda_i \in \ko$ such that
$M_i \cong \{(\lambda_i x_i, x_i) \mid x_i \in V_i\} \le V_0 \oplus V_i$.
Since $r \vert_M \hskip.1cm \colon M \to V_1 \oplus \dots \oplus V_\ell$ is an isomorphism,
we deduce that 
\begin{align*}
M & = M_1 \oplus \dots \oplus M_\ell
\\
& = \Big\{ \Big( \sum_{i=1}^\ell \lambda_i x_i, x_1, x_2, \dots, x_\ell \Big)
\hskip.2cm \Big\vert \hskip.2cm
(x_1, \dots, x_\ell) \in V_1 \oplus \dots \oplus V_\ell \Big\}
\end{align*}
as required. 
\end{proof}

We can now characterize
when a direct sum of copies of a given simple $\F_p[G]$-module is cyclic.

\begin{lem}
\label{lem:Isotypical}
Retain the notation of Lemma \ref{lem:Goursat-for-ell+1}.
\par

The $\F_p[G]$-module $U$ is cyclic if and only if $\ell < \dim_{\ko} (W)$. 
\end{lem}

\begin{proof}
Assume first that $\ell \ge \dim_{\ko} (W)$.
Let $(v_0, \dots, v_\ell) \in U$.
Since $V_i = W$ for all $i$, we may view $v_i$ as an element of $W$.
Then, upon reordering the summands $V_0, \dots, V_\ell$,
we may assume that there exists $(\lambda_1, \dots, \lambda_\ell) \in \ko^\ell$
such that $v_0 = \sum_{i=1}^\ell \lambda_i v_i$.
It follows that $(v_0, \dots, v_\ell)$ belongs to 
$$
\Big\{ \Big( \sum_{i=1}^\ell \lambda_i x_i, x_1, x_2, \dots, x_\ell \Big) 
\hskip.2cm \Big\vert \hskip.2cm
(x_1, \dots, x_\ell) \in V_1 \oplus \dots \oplus V_\ell \Big\},
$$
which is a proper $\F_p[G]$-submodule of $U$. Hence $U$ is not cyclic.
(In a context of characteristic zero, an argument of this kind is used
for the proof of \cite[Lemma~15.5.3]{Dixm--69}.)
\par

In order to prove the converse, we proceed by induction on $\ell$.
In case $\ell =0$, we have $0 = \ell < \dim_{\ko} (W)$
and $U = V_0 = W$ is simple, hence cyclic.
\par

We now assume that $0 < \ell < \dim_{\ko} (W)$.
The induction hypothesis ensures that the $\F_p[G]$-module $V_1 \oplus \dots \oplus V_\ell$ is cyclic.
Let $(v_1, \dots, v_\ell)$ be a generator.
Viewing all $v_i$ as elements of $W$, the hypothesis that $\ell < \dim_{\ko} (W)$
ensures the existence of an element $v_0 \in W$
which does not belong to the $\ko$-subspace of $W$ spanned by $\{v_1, \dots, v_\ell\}$.
Let $M$ be the $\F_p[G]$-submodule of $U$ spanned by $(v_0, v_1, \dots, v_\ell)$.
Let $r \hskip.1cm \colon U \twoheadrightarrow V_1 \oplus \dots \oplus V_\ell$
denote the canonical projection.
The image $r(M)$ coincides with the $\F_p[G]$-submodule
generated by $(v_1, \hdots, v_\ell)$, i.e., with $V_1 \oplus \dots \oplus V_\ell$.
If one had $M \cap V_0 = \{0\}$,
then $M$ would be a maximal proper $\F_p[G]$-submodule of $U$,
and Lemma~\ref{lem:Goursat-for-ell+1} would then ensures that
$v_0$ is a $\ko$-linear combination of $\{v_1, \dots, v_\ell\}$, a contradiction.
Hence $M \cap V_0 \ne \{0\}$. Since $V_0$ is simple, $M$ contains $V_0$,
so that $U = M$; this shows that $U$ is indeed cyclic. 
\end{proof}

The following basic counting lemma will also be useful. 

\begin{lem}
\label{lem:counting}
Retain the notation of Lemma \ref{lem:Goursat-for-ell+1}.
Moreover, set $q = \vert \ko \vert$.
\par

The number of simple $\F_p[G]$-submodules of $U$ is 
$$
q^\ell + q^{\ell-1} + \dots + q + 1 = \frac{q^{\ell+1}-1}{q-1}. 
$$
\end{lem}

\begin{proof}
We proceed by induction on $\ell$.
In case $\ell = 0$, the $\F_p[G]$-module $U = V_0$ is simple, so the result is clear. 
Assume now that $\ell \ge 1$.
Consider
\par
the collection $\text{\large$\mathcal S$}$
of all simple $\F_p[G]$-submodules of $U$,
\par
the complement $\text{\large$\mathcal S$}_0$
of $V_0$ in $\text{\large$\mathcal S$}$, 
i.e., $\text{\large$\mathcal S$}_0 = 
\{ S \in \text{\large$\mathcal S$} \mid S \cap V_0 = \{0\} \}$,
\par
and the collection $\text{\large$\mathcal S'$}$
of all simple $\F_p[G]$-submodules of $V_1 \oplus \dots \oplus V_\ell$.
\par\noindent
Denote by $r$ the canonical projection
$U \twoheadrightarrow V_1 \oplus \dots \oplus V_\ell$.
Each $S \in \text{\large$\mathcal S$}_0$ determines
its image $S' = r(S) \in \text{\large$\mathcal S'$}$,
which can be viewed as a submodule of $U$
contained in $V_1 \oplus \dots \oplus V_\ell$,
and there exists by Lemma \ref{lem:Goursat} an element
$\lambda \in \ko$
such that $S = \{(\lambda x, x) \mid x \in S'\} \le V_0 \oplus S'$.
Conversely, $S' \in \text{\large$\mathcal S'$}$ and $\lambda \in \ko$
determine $S$.
This shows that
$\vert \text{\large$\mathcal S$} \vert
= \vert \text{\large$\mathcal S$}_0 \vert + 1
= q \vert \text{\large$\mathcal S'$} \vert + 1$.
Since $\vert \text{\large$\mathcal S'$} \vert = q^{\ell-1} + \dots + q + 1$
by the induction hypothesis, this ends the proof.
\end{proof}

\begin{lem}
\label{lem:CyclicQuotientModule}
Retain the notation of Lemma \ref{lem:Goursat-for-ell+1}.
Moreover, set $q = \vert \ko \vert$,
denote by $m$ the dimension of $W$ over $\ko$,
and assume that $\ell \ge m$.
Let $\mathcal Z$ be a set of simple $\F_p[G]$-submodules of $U$
of cardinality $\vert \mathcal Z \vert < q^m + \dots + q + 1$.
\par

There is an $\F_p[G]$-submodule $B \le U$ with $B \cap Z = \{0\}$ for all $Z \in \mathcal Z$,
and such that $U/B$ is cyclic.
\end{lem}

\begin{proof}
Let $\text{\large$\mathcal B$}$ be the collection of all $\F_p[G]$-submodules $B$ of $U$
such that $B \cap Z = \{0\}$ for all $Z \in \mathcal Z$.
Let also $B \in \text{\large$\mathcal B$}$ be an element which is maximal for the inclusion relation.
Note that the $\F_p[G]$-module $U/B$ is semi-simple.
If $U/B$ were not cyclic,
then $U/B$ would be isomorphic to a direct sum of at least $m+1$ copies of $W$ by Lemma~\ref{lem:Isotypical}.
Therefore $U/B$ would contain at least $q^m + \dots + q + 1$
simple $\F_p[G]$-submodules by Lemma~\ref{lem:counting}.
In particular $U/B$ would contain at least one simple $\F_p[G]$-submodule $C$
which is different from the canonical image of $Z$ in $U/B$ for all $Z \in \mathcal Z$.
Denoting by $B'$ the preimage of $C$ in $U$,
we obtain $B \lneqq B'$ and $B' \in \text{\large$\mathcal B$}$.
This contradicts the maximality of $B$. Hence $U/B$ is cyclic.
\end{proof}

Given an additive group $V$ and a subset $F \subseteq V$, we set 
$$
F - F = \{c \in V \mid c = a-b \hskip.2cm \text{for some} \hskip.2cm a, b \in F \} .
$$
The following result will be needed in Section~\ref{SectionQ(n)}.

\begin{lem}
\label{lem:F-F}
Retain the notation from Lemma~\ref{lem:counting}.
\par

If $\ell \ge 1$,
there is a subset $F \subseteq U$ of size~$q^\ell + q^{\ell-1} + \dots + q + 1$
such that $F-F$ contains a non-zero element
of each of the simple $\F_p[G]$-submodules of $U$.
\end{lem}

\begin{proof}
For each subset $I \subseteq \{0, 1, \dots, \ell\}$,
we view the direct sum $\bigoplus_{i \in I} V_i$ as a submodule of $U$. 
\par
 	
Let $\text{\large$\mathcal S$}$ be the collection of all simple submodules of $U$
and $\text{\large$\mathcal S$}'$ be the subcollection consisting of
those $S \in \text{\large$\mathcal S$}$ which are contained in $V_0 \oplus V_1$.
By Lemma~\ref{lem:counting}, we have
$\vert \text{\large$\mathcal S$} \vert = q^\ell + q^{\ell-1} + \dots + q + 1$
and $\vert \text{\large$\mathcal S$}' \vert = q+1$.
\par

By definition, each element of $\text{\large$\mathcal S$}$ is a simple module,
so that any two distinct elements of $\text{\large$\mathcal S$}$ have intersection $\{0\}$.
Choosing a non-zero element in each member of
$\text{\large$\mathcal S$} \smallsetminus \text{\large$\mathcal S$}'$,
we obtain a set $E$ of size $q^\ell + q^{\ell-1} + \dots + q^2$.
\par

Choose now a non-zero $x \in V_1$, and set
$E' = \{(\lambda x, x) \mid \lambda \in k\} \subseteq V_0 \oplus V_1$.
Thus $\vert E' \vert = \vert \ko \vert = q$.
By Lemma~\ref{lem:Goursat-for-2}, the set $E'$ contains a non-zero element
in each member of $\text{\large$\mathcal S$}' \smallsetminus \{V_0\}$.
\par

Finally, we set $F = E \cup E' \cup \{0\}$.
Observe that $\vert F \vert = q^\ell + q^{\ell-1} + \dots + q + 1$.
Moreover, we have $E \cup E' \subset F \smallsetminus \{0\} \subset F-F$,
so that $F-F$ contains a non-zero element
of each member of $\text{\large$\mathcal S$} \smallsetminus \{V_0\}$.
Since
$$
V_0 \ni (x, 0) = (x, x)-(0, x) \in E' - E' \subset F-F,
$$
we see that $F-F$ also contains a non-zero element of $V_0$.
Thus the set $F$ has the required properties. 
\end{proof}

Given a semi-simple $R$-module $U$ and a simple $R$-module $W$,
the submodule of $U$ generated by all simple submodules isomorphic to $W$
is called the \textbf{isotypical component of type $W$} of $U$.
Every semi-simple $R$-module is the direct sum of its isotypical components
\cite[\S~3, Proposition~9]{Bo--A8}.

\begin{lem}
\label{lem:IsotypicalDecomposition}
A finite semi-simple $\F_p[G]$-module $U$ is cyclic
if and only if each of its isotypical components is cyclic. 
\end{lem}

\begin{proof}
The `only if' part is clear since any quotient of a cyclic module is cyclic. 
\par

Let $U = M_1 \oplus \dots \oplus M_\ell$ be the decomposition of $U$
as the direct sum of its isotypical components. 
Assume that $M_i$ is cyclic for all $i \in \{1, \dots, \ell\}$ and let $v_i \in M_i$ be a generator.
We claim that $v = (v_1, \dots, v_\ell)$ is a generator of $U$.
We prove this by induction on $\ell$.
The base case $\ell = 1$ is trivial.
Assume now that ${\ell \ge 2}$ and let $M$ be the submodule generated by $v$.
The induction hypothesis ensures that
the canonical projection of $M$ to $A_0 = \bigoplus_{i=1}^{\ell-1} M_i$ is surjective.
Clearly, the projection of $M$ to $A_1 = M_\ell$ is surjective.
Since $A_0$ and $A_1$ are disjoint
(i.e., they do not contain any non-zero isomorphic summands),
it follows from Lemma~\ref{lem:Goursat} that $M = A_0 \oplus A_1 = U$.
\end{proof}

\begin{prop}
\label{prop:cyclicsemi-simple}
Let $U$ be a finite semi-simple $\F_p[G]$-module.
The following properties are equivalent:
\begin{enumerate}[label=(\roman*)]
\item \label{iDEprop:cyclicsemi-simple}
$U$ is not cyclic. 
\item\label{iiDEprop:cyclicsemi-simple}
There exist a finite simple $\F_p[G]$-module $W$
of dimension $m \ge 1$ over $\ko = \mathcal L_{\F_p[G]}(W)$,
and a submodule $V \le U$ isomorphic to a direct sum of $m+1$ copies of $W$.
\end{enumerate}
\end{prop}

\begin{proof}
Assume that Property \ref{iiDEprop:cyclicsemi-simple} holds.
In view of Lemma~\ref{lem:Isotypical},
the module $V$ afforded by \ref{iiDEprop:cyclicsemi-simple} is not cyclic.
Since $V$ is a direct summand of $U$,
and therefore isomorphic to a quotient of $U$,
it follows that $U$ is not cyclic. 
\par

Assume conversely that $U$ is not cyclic.
Then $U$ has a non-cyclic isotypical component
by Lemma~\ref{lem:IsotypicalDecomposition},
and it follows from Lemma~\ref{lem:Isotypical}
that Condition \ref{iiDEprop:cyclicsemi-simple} holds.
\end{proof}

\section{On the structure of minimal unfaithful subsets}
\label{Section:proofof1.1}

The goal of this section is to prove Theorem~\ref{thm:CharP(n)}.
In fact, we shall establish a finer statement
that describes precisely the structure of the normal closure
of an irreducibly unfaithful subset of size~$n$ in a countable group with Property $P(n-1)$,
see Theorem~\ref{thm:UnfaithfulSetSize-n} below.
We shall however start
with the proof of the easier implication in Theorem~\ref{thm:CharP(n)}.

\subsection{Proof of \ref{2DEthm:CharP(n)} 
\texorpdfstring{$\Rightarrow$}{=>}
\ref{1DEthm:CharP(n)}
in Theorem~\ref{thm:CharP(n)}}

\begin{lem}
\label{firstpartproofMainTheorem}
Let $G$ be a countable group.
Suppose that there exist a prime $p$,
a~finite normal subgroup $V$ of $G$ which is an elementary abelian $p$-group,
and a finite simple $\F_p[G]$-module $W$,
with centralizer field $\ko = \mathcal L_{\F_p[G]}(W)$ and dimension $m = \dim_{\ko} (W)$,
such that $V$ is isomorphic as $\F_p[G]$-module to a direct sum of $m+1$ copies of~$W$.
Set $q = \vert \ko \vert$. Then:
\begin{enumerate}[label=(\roman*)]
\item\label{iDEfirstpartproofMainTheorem}
For every irreducible unitary representation $\pi$ of $G$,
the kernel $\Ker(\pi)$ contains at least one
of the $q^m + \dots + q + 1 $ simple submodules of $V$.
\item\label{iiDEfirstpartproofMainTheorem}
A subset $F \subset V$ is irreducibly faithful in $G$
if and only if $F \cap K \subset\{e\}$ for some
simple $\F_p[G]$-submodule $K$ of $V$.
\item\label{iiiDEfirstpartproofMainTheorem}
There is a subset $F \subset V$ of size $q^m + \dots + q + 1 $
which is not irreducibly faithful.
\end{enumerate}
\end{lem}

\begin{proof}
\ref{iDEfirstpartproofMainTheorem}
Since $V$ is not cyclic as an $\F_p[G]$-module by Lemma~\ref{lem:Isotypical},
it follows that $V$ has no $G$-faithful character by Proposition~\ref{lem14BEHA}.
In view of Lemma~\ref{FromL9}, 
for every irreducible unitary representation $\pi$ of $G$,
the restriction $\pi \vert_V$ cannot be faithful.
In particular $\Ker(\pi)$ contains at least one of the simple $\F_p[G]$-submodules of $V$.
\par

\ref{iiDEfirstpartproofMainTheorem}
Let $F$ be a subset of $V$.
If $F$ contains a non-trivial element in each of the simple $\F_p[G]$-submodules of $V$,
it follows from \ref{iDEfirstpartproofMainTheorem}
that $F$ is not irreducibly faithful in $G$.
Conversely,
if there is a simple $\F_p[G]$-submodule $K$ in $V$ such that $F \cap K \subset \{e\}$,
then every non-trivial element of $F$ has a non-trivial image in the quotient $V/K$.
Since $V/K$ is a cyclic $\F_p[G]$-module by Lemma~\ref{lem:Isotypical},
it follows from Proposition~\ref{lem14BEHA} and Lemma~\ref{lem:G-faithful-normal-subgroup}
that $G/K$ has an irreducible unitary representation whose restriction to $V/K$ is faithful.
Terefore $F$ is irreducibly faithful in $G$.
\par

\ref{iiiDEfirstpartproofMainTheorem}
By Lemma~\ref{lem:counting},
the number of simple submodules in $V$ equals $q^m + \dots + q + 1$.
Thus $V$ contains a subset of size $q^m + \dots + q + 1 $ which is not irreducibly faithful. 
\end{proof}

\begin{proof}[Proof of \ref{2DEthm:CharP(n)} $\Rightarrow$ \ref{1DEthm:CharP(n)}
in Theorem~\ref{thm:CharP(n)}]
If $G$ satisfies \ref{2DEthm:CharP(n)} in Theorem~\ref{thm:CharP(n)},
then $G$ contains a set $F \subseteq V$ of size $q^m + \dots + q + 1$
which is not irreducibly faithful by Lemma~\ref{firstpartproofMainTheorem}.
so that $G$ does not have Property $P( q^m + \dots + q + 1)$,
Since $n \ge q^m + \dots + q + 1$, the group $G$ does not have Property $P(n)$.
\end{proof}

\subsection{Minimal unfaithful subsets are contained in \texorpdfstring{$\Fsol (G)$}{FSol (G)}}

The following result shows that, in a countable group $G$,
the irreducible faithfulness of a subset $F$ can be checked
on the intersection of $F$ with $\Fsol (G)$.

\begin{prop}
\label{prop:IrredFaithful:sol}
Let $G$ be a countable group 
and $F$ a subset of $G$.
If $F \cap \Fsol (G)$ is finite and irreducibly faithful, then $F$ is irreducibly faithful. 
\par

In particular, a finite subset $F \subseteq G$ is irreducibly faithful
if and only if the intersection $F \cap \Fsol (G)$ is irreducibly faithful. 
\end{prop}

\emph{Note~:} 
We know already the particular case of this proposition for $F$ disjoint from $\Fsol (G)$,
for example for $F = G \smallsetminus \Fsol (G)$, see Remark \ref{TrunkSufFiur}(1).

\begin{proof}
Set $S = \Fsol (G)$. 
Let $F \subseteq G$ be such that $F \cap S$ is finite and irreducibly faithful.
We aim at proving that $F$ is irreducibly faithful.
To this end, we partition $F$ into three subsets,
$F = F_S \sqcup F_H \sqcup F_\infty$, where:
\begin{enumerate}
\item[]
$F_S = \{ x \in F \mid 
\llag x \rrag_G \hskip.2cm \text{is finite soluble} \} = F \cap S$,
\item[]
$F_H = \{ x \in F \mid 
\llag x \rrag_G \hskip.2cm \text{is finite non-soluble} \} = (F \cap \W(G)) \smallsetminus S$,
\item[]
$F_\infty = \{ x \in F \mid \llag x \rrag_G \hskip.2cm \text{is infinite} \}
= F \smallsetminus \left( F_S \sqcup F_H \right)$.
\end{enumerate}
\par

By hypothesis, there exists an irreducible unitary representation $\rho$ of $G$
such that $\rho(x) \neq \id$ for all $x \in F_S \smallsetminus \{e\}$.
Since $F_S$ is finite by hypothesis,
the normal subgroup $A = \llag F_S \rrag_G$ is finite soluble
(Lemma \ref{lem:N1Nk}).
Let $K = A \cap \Ker(\rho)$, which is a finite soluble normal subgroup of $G$,
and let $r \hskip.1cm \colon G \twoheadrightarrow Q = G / K$ be the canonical projection.
Note that $r(x) \neq e$ for all $x \in F_S \smallsetminus \{e\}$.
\par

Since $A$ is soluble, its image $\rho(A)$ is soluble as well.
Therefore the socle $\Soc(\rho(A))$ is abelian. 
Since $\rho(G)$ is irreducibly faithful,
the socle $\Soc(\rho(A))$
has a $\rho(G)$-faithful irreducible unitary character by Lemma~\ref{FromL9}.
\par

The homomorphism $\rho$ induces
an isomorphism $\rho_A \hskip.1cm \colon A / K \overset{\cong}{\longrightarrow} \rho(A)$,
and similarly $r$ induces
an isomorphism $r_A \hskip.1cm \colon A / K \overset{\cong}{\longrightarrow} r(A)$.
Moreover, the action by conjugation of $G$ on $A$
induces $G$-actions on $\rho(A)$ and $r(A)$,
and the isomorphism $r_A \rho_A^{-1} \hskip.1cm \colon \rho(A) \overset{\cong}{\longrightarrow} r(A)$
is $G$-equivariant.
Hence the group $N = \Soc(r(A))$, which is normal in $Q$,
is abelian and has a $Q$-faithful unitary character, say $\sigma$.
\par

We now invoke Lemma~\ref{lem:G-faithful-normal-subgroup:Kernel<sol(G)},
which affords an irreducible unitary representation $\pi$ of $Q$
whose restriction to $N$ is faithful,
and such that $\Ker(\pi)$ is contained in $\Fsol (Q)$.
\par 

The composite map $\pi' = \pi \circ r$
is an irreducible unitary representation of $G$.
\par 

We claim that $\pi'(x) \neq \id$
for all $x \in \big( F_S \smallsetminus \{e\} \big)$.
We know that the representation $\pi \vert_N$ is faithful.
Since $N = \Soc (r(A))$,
it follows that $\pi \vert_{r(A)}$ is also faithful.
As noted above,
for every $x \in F_S \smallsetminus \{e\}$,
we have $r(x) \ne e$, and therefore $\pi'(x) = \pi(r(x)) \ne \id$.
\par

We next claim that $\pi'(x) \neq \id$ for all $x \in F_H$.
Indeed, for $x \in F_H$, we have
$x \neq e$ and $\llag x \rrag_G \not \le S$,
since $\llag x \rrag_G$ is not soluble.
But $K \cap \llag x \rrag_G$ is finite soluble, because $K$ is so.
Therefore
$ \llag r(x) \rrag_Q = r(\llag x \rrag_G) \cong \llag x \rrag_G / (K \cap \llag x \rrag_G)$
is not soluble, hence $r(x) \notin \Fsol (Q)$.
Since the kernel of $\pi$ is contained in $\Fsol (Q)$,
this shows that $r(x) \notin \Ker (\pi)$, and therefore that $x \notin \Ker (\pi')$,
as claimed.
\par 

Given $x \in F_\infty$, the normal closure $\llag x \rrag_G$ is infinite.
Since $K$ is finite, it follows that $\llag r(x) \rrag_Q$ is infinite as well.
In particular $r(x)$ is not contained in the kernel of $\pi$,
which is contained in $\Fsol (Q) \le \W(Q)$. Hence $\pi'(x) \neq 1$.
This proves that $\pi'(x) \neq 1$ for all $x \in F \smallsetminus \{e\}$.
\par

Thus $F$ is irreducibly faithful, as required. 
\end{proof}

\begin{cor}
\label{cor:MinimalIrreduciblyUnfaithful}
Let $G$ be a countable group
and $F \subseteq G$ be a finite subset which is irreducibly unfaithful.
\par

If every proper subset of $F$ is irreducibly faithful,
then $F$ is contained in $\Fsol (G)$.
\end{cor}

\begin{proof}
Let $F$ be a finite subset of $G$ of which every proper subset is irreducibly faithful.
If $F$ was not contained in $\Fsol (G)$, then $F$ would be irreducibly faithful by Proposition~\ref{prop:IrredFaithful:sol}.
Therefore $F \subseteq \Fsol (G)$.
\end{proof}

\subsection{Unfaithful subsets of size \texorpdfstring{$n$}{n}
in countable groups with \texorpdfstring{$P(n-1)$}{P(n-1)}}

\begin{lem}
\label{lem:UnfaithfulSubsets+mini-feet}
Let $n$ be a positive integer and $G$ a countable group.
Let $F \subset G$ be an irreducibly unfaithful subset of size $n$ such that:
\begin{enumerate}[label=(\alph*)]
\item\label{aDElem:UnfaithfulSubsets+mini-feet}
every proper subset of $F$ is irreducibly faithful;
\item\label{bDElem:UnfaithfulSubsets+mini-feet}
every element of $F$ is contained in an abelian mini-foot of $G$.
\end{enumerate}
Let $U = \llag F \rrag_G$ the normal subgroup of $G$ generated by $F$.
\par

Then there exist a prime $p$,
and a simple $\F_p[G]$-module $W$,
such that the following assertions hold,
where $\ko$ denotes the centralizer field $\mathcal L_{\F_p[G]}(W)$,
and $m = \dim_{\ko}(W)$, and $q = \vert \ko \vert$~:
\begin{enumerate}[label=(\roman*)]
\item\label{iDElem:UnfaithfulSubsets+mini-feet}
$U$ is a finite elementary abelian $p$-group,
 contained in the abelian mini-socle $\MA (G)$;
\item\label{iiDElem:UnfaithfulSubsets+mini-feet}
$U$ is isomorphic as an $\F_p[G]$-module
to the direct sum of a number $\ell+1$ of copies of $W$, and $\ell \ge m$;
\item\label{iiiDElem:UnfaithfulSubsets+mini-feet}
$q^m + q^{m-1} + \dots + q + 1 \le n$. 
\end{enumerate}
\end{lem}

\begin{proof}
The hypothesis \ref{aDElem:UnfaithfulSubsets+mini-feet} on $F$
implies that $F$ does not contain the neutral element $e$, since otherwise $F$ would be irreducibly faithful. 
\par

By Proposition \ref{structureminisocle},
the normal subgroup $U$ is abelian and finite.
The conjugation action of $G$ on $U$ allows us to view $U$ as a $\Z[G]$-module.
Since $U$ is generated by mini-feet of $G$, 
it follows from Proposition~\ref{prop:semi-simple} that $U$ is a semi-simple $\Z[G]$-module.
Let $\mathcal Y$ denote the set of isomorphism classes
of simple $\Z[G]$-submodules of $U$.
For each $Y \in \mathcal Y$, let $U_Y$ be the submodule of $U$
generated by the simple submodules isomorphic to $Y$;
note that $U_Y$ is a finite abelian normal subgrouop of $U$.
We have the isotypical direct sum decomosition
$U = \bigoplus_{Y \in \mathcal Y} U_Y$.
\par 

For $x \in F$, the normal closure $\llag x \rrag_G \le U$
is an abelian mini-foot of $G$ by hypothesis, hence a simple $\Z[G]$-module.
Thus it is isomorphic to some $Y \in \mathcal Y$ and $\llag x \rrag_G \le U_Y$.
Setting $F_Y = F \cap U_Y$ for all $Y \in \mathcal Y$,
we obtain a partition of $F$ as $F = \bigsqcup_{Y \in \mathcal Y} F_Y$. 

\vskip.2cm

We claim that $\mathcal Y$ contains a single element.
Indeed, assume this is not the case.
For each $Y \in \mathcal Y$,
the subset $F_Y$ is striclty contained in $F$,
hence is irreducibly faithful by the hypotheses made on $F$.
Let $\pi_Y$ be an irreducible unitary representation of $G$
witnessing the faithfulness of $F_Y$,
and set $K_Y = U_Y \cap \Ker(\pi_Y)$.
Thus every element of $F_Y$ has a non-trivial image in $U_Y / K_Y$.
Moreover, we may view $\pi_Y$ as an irreducible unitary representation of $G / K_Y$
whose restriction to $U_Y / K_Y$ is faithful.
By Lemma~\ref{FromL9},
$U_Y / K_Y$ has a $G / K_Y$-faithful unitary character.
Therefore it is a cyclic $\F_p[G / K_Y]$-module by Proposition~\ref{lem14BEHA},
where $p$ is the exponent of $Y$. In particular it is a cyclic $\Z[G]$-module
\par

Let now $K = \left\langle \bigcup_{Y \in \mathcal Y} K_Y \right\rangle$.
Thus $K$ is a normal subgroup of $G$,
and we have a natural direct sum decomposition $K \cong \bigoplus_{Y \in \mathcal Y} K_Y$
(see Lemma \ref{lem:N1Nk}).
In particular $K \cap U_Y = K_Y$ for all $Y \in \mathcal Y$.
Moreover, we have $K \cap F = \varnothing$,
since otherwise $K \cap F_Y$ would be non-empty for some $Y \in \mathcal Y$,
which would imply that $K_Y$ contains an element of $F_Y$.
This contradicts the definition of $K_Y$.
Therefore, every element of $F$ has a non-trivial image in $G / K$.
\par

We may view the quotient $U / K$ as a $\Z[G]$-module.
It is semi-simple by Proposition~\ref{prop:semi-simple}, as a quotient module of $U$.
Moreover, the direct sum decomposition $U / K \cong \bigoplus_{Y \in \mathcal Y} U_Y / K_Y$
is the isotypical decomposition of $U / K$.
We have seen above that the isotypical component $U_Y / K_Y$ of $U / K$
is a cyclic $\Z[G]$-module for each $Y \in \mathcal Y$. 
It follows from Lemma~\ref{lem:IsotypicalDecomposition} that $U / K$ is cyclic.
By Proposition~\ref{lem14BEHA},
this means that $U / K$ has a $G / K$-faithful unitary character.
By Lemma~\ref{lem:G-faithful-normal-subgroup},
$G / K$ has an irreducible unitary representation $\pi$
whose restriction to $U / K$ is faithful.
Since every element of $F$ has a non-trivial image in $G / K$,
precomposing $\pi$ with the projection $G \twoheadrightarrow G / K$
yields an irreducible unitary representation of $G$
mapping every element of $F$ to a non-trivial operator.
Thus $F$ is irreducibly faithful, a contradiction. This proves the claim. 

\vskip.2cm

We denote the single element of $\mathcal Y$ by $W$,
From now on, denote by $p$ the exponent of $W$.
Since the $\Z[G]$-module $W$ is simple, $p$ is a prime. 
Thus $W$ is a simple $\F_p[G]$-module,
and $U = U_W$ is isomorphic to a direct sum of $\ell+1$ copies of $W$
for some integer $\ell \ge 0$.
Since $F$ is not irreducibly faithful,
it follows that the restriction to $U$ of every irreducible unitary representation of $G$
cannot be faithful.
Therefore $U$ has no $G$-faithful character by Lemma~\ref{lem:G-faithful-normal-subgroup}.
Hence $U$ is not a cyclic $\F_p[G]$-module by Proposition~\ref{lem14BEHA}.
In view of Proposition \ref{prop:cyclicsemi-simple},
this implies that $\ell \ge m$,
where $m$ is the dimension of $W$ over $\ko = \mathcal L_{\F_p[G]}(W)$.
This proves \ref{iDElem:UnfaithfulSubsets+mini-feet} and \ref{iiDElem:UnfaithfulSubsets+mini-feet}. 

\vskip.2cm

It remains to prove that $n = \vert F \vert \ge q^m + \dots + q +1$.
Recall from the hypothesis that
$\llag x \rrag_G $ is a simple $\F_p[G]$-submodule of $U$ for all $x \in F$. 
Assume for a contradiction that $n < q^m + \dots + q +1$.
Then, by Lemma~\ref{lem:CyclicQuotientModule},
there is an $\F_p[G]$-submodule $B \le U$
with $B \cap \llag x \rrag_G = \{e\}$ for all $x \in F$, such that $U/B$ is cyclic.
It then follows from Lemma~\ref{lem:G-faithful-normal-subgroup} and Proposition~\ref{lem14BEHA}
that $G/B$ has an irreducible unitary representation whose restriction to $U/B$ is faithful.
Viewing that representation as a representation of $G$,
we obtain a contradiction with the fact that $F$ is irreducibly unfaithful.
This proves \ref{iiiDElem:UnfaithfulSubsets+mini-feet}.
\end{proof}

In the introduction, Property $P(n)$ was defined for all $n \ge 1$. 
For the sake of uniformity in the forthcoming arguments, 
we extend the definition to the case $n=0$.
Thus every group has Property $P(0)$, tautologically.
The main result of this section is the following. 

\begin{thm}
\label{thm:UnfaithfulSetSize-n}
Let $n$ be a positive integer and $G$ a countable group with Property $P(n-1)$.
Let $F \subset G$ be an irreducibly unfaithful subset of size $n$,
and $U = \llag F \rrag_G$ the normal subgroup of $G$ generated by $F$.
\par

Then there exist a prime $p$
and a finite simple $\F_p[G]$-module $W$
such that the following assertions hold,
where $\ko$ denotes the centralizer field $\mathcal L_{\F_p[G]}(W)$,
and $m = \dim_{\ko}(W)$, and $q = \vert \ko \vert$~:
\begin{enumerate}[label=(\roman*)]
\item\label{iDEthm:UnfaithfulSetSize-n}
$U$ is a finite elementary abelian $p$-group,
contained in the abelian mini-socle $\MA (G)$;
\item\label{iiDEthm:UnfaithfulSetSize-n}
$U$ is isomorphic as an $\F_p[G]$-module
to the direct sum of a number $\ell+1$ of copies of $W$, and $\ell \ge m$;
\item\label{iiiDEthm:UnfaithfulSetSize-n}
$q$ is a power of $p$ and
$q^m + q^{m-1} + \dots + q + 1 = n$.
\end{enumerate}
\end{thm}

\begin{proof}
It follows from the hypotheses that
every proper subset of $F$ is irreducibly faithful; in particular $e \not \in F$.
Hence $F \le \Fsol (G)$ by Corollary~\ref{cor:MinimalIrreduciblyUnfaithful}.
For every $x \in F$, it follows that the group $\llag x \rrag_G$ is finite soluble,
and therefore has an abelian socle.
Since socles are characteristic subgroups,
this socle is a finite abelian normal subgroup of $G$,
hence it contains an abelian mini-foot of $G$.
We may therefore choose $b_x \in \llag x \rrag_G$
such that $\llag b_x \rrag_G$ is an abelian mini-foot of $G$. 
\par

We set $F' = \{b_x \mid x \in F\}$, so that $\vert F' \vert \le \vert F \vert = n$.
Since $b_x \in \llag x \rrag_G$,
we see that $F'$ is irreducibly unfaithful, because $F$ itself has that property.
If $\vert F' \vert < n$,
then $F'$ would be faithful since $G$ has $P(n-1)$, a contradiction.
Thus $\vert F' \vert = n$, and every proper subset of $F'$ is irreducibly faithful.
We may therefore apply Lemma~\ref{lem:UnfaithfulSubsets+mini-feet} to the set $F'$.
We denote by $p$, $U' = \llag F' \rrag_G$, $W$, $m$, $q$
the various objects afforded in that way. 
Then Properties
\ref{iDEthm:UnfaithfulSetSize-n} and \ref{iiDEthm:UnfaithfulSetSize-n}
are satisfied by $U'$.
Moreover we have $q^m + q^{m-1} + \dots + q + 1 \le n$.
If we had $q^m + q^{m-1} + \dots + q + 1 \le n-1$,
then $G$ would not have Property $P(n-1)$ by Lemma~\ref{firstpartproofMainTheorem}.
Therefore Property~\ref{iiiDEthm:UnfaithfulSetSize-n}
is also satisfied by $U'$.
It remains to show that $U' = U$, i.e., that $U' = \llag F \rrag_G$. 

\vskip.2cm

Since $b_x \in \llag x \rrag_G$ for all $x \in F$,
we have $U' = \llag F' \rrag_G \le \llag F \rrag_G$.
Thus it suffices to show that $F$ is contained in $U'$.
Assume for a contradiction that this is not the case,
and let $y \in F$ be such that $y \notin U'$.
Since $F' \smallsetminus \{b_y\}$ is irreducibly faithful,
there exists an irreducible unitary representation $\rho$ of $G$
such that $\rho(b_x) \neq 1$ for all $x \in F \smallsetminus \{y\}$.
Let $B = U' \cap \Ker(\rho)$.
By Lemma~\ref{FromL9} and Proposition~\ref{lem14BEHA},
the $\F_p[G]$-module $U'/B$ is cyclic.
In particular, it is isomorphic to a direct sum of $j$ copies of $W$,
for some $j \in \{1, \dots, m\}$, by Lemma~\ref{lem:Isotypical}. 
\par

Let $r \hskip.1cm \colon G \twoheadrightarrow G/B = Q$ be the canonical projection.
We have seen that $r(U') = U'/B$ is a cyclic $\F_p[Q]$-module.
Thus $U'/B$ has a $Q$-faithful unitary character by Proposition~\ref{lem14BEHA}. 
\par 

Since $y \not \in U'$, we have $r(y) \notin r(U')$.
Since $F$ is contained in $\Fsol (G)$,
it follows that $\llag r(y) \rrag_Q$ is soluble finite, i.e., $r(y) \in \Fsol (Q)$.
In particular the socle of $\llag r(y) \rrag_Q$ is abelian. 
We may therefore choose $b'_y \in \llag r(y) \rrag_Q$
such that $\llag b'_y \rrag_Q$ is an abelian mini-foot of $Q$,
not contained in $r(U')$.
Now we discuss the structure of $\llag b'_y \rrag_Q$,
in order to achieve a contradiction.

\vskip.2cm

Since $\llag b'_y \rrag_Q$ is a mini-foot of $Q$,
it follows that $\llag b'_y \rrag_Q$ may be viewed as a simple $\Z[Q]$-module
(see Remark \ref{remabnormalmodule}).
In particular the normal subgroup 
$$
N = r(U') \times \llag b'_y \rrag_Q
$$
is a semi-simple $\Z[Q]$-module, by Proposition~\ref{prop:semi-simple}.
\par

We claim that $N$ is not cyclic as a $\Z[Q]$-module.
Suppose by contradiction that $N$ is cyclic.
Then $N$ has a $Q$-faithful unitary character by Proposition~\ref{lem14BEHA}. 
Therefore $Q$ has an irreducible unitary representation
whose restriction to $N$ is faithful,
by Lemma~\ref{lem:G-faithful-normal-subgroup}.
It follows that the set $r(F' \smallsetminus \{b_y\}) \cup \{b'_y\}$
is irreducibly faithful in~$Q$.
Notice that,
if the kernel of a unitary representation of $Q$ contains $r(x)$
for some $x$ in $F \smallsetminus \{y\}$,
then it contains $r(b_x)$ since $b_x \in \llag x \rrag_G$.
Similarly, if that kernel contains $r(y)$, then it contains $b'_y$. 
Since $r(F' \smallsetminus \{b_y\}) \cup \{b'_y\}$ is irreducibly faithful in $Q$,
we infer that the set $r(F)$ is irreducibly faithful in $Q$.
Hence $F$ is irreducibly faithful in $G$, a contradiction. 
This confirms that the $\Z[Q]$-module $N$ is not cyclic.
\par

Since $Q$ is a quotient of $G$, we may view any $\Z[Q]$-module as a $\Z[G]$-module.
We have seen above that $r(U')$ is a cyclic $\F_p[Q]$-module,
hence a cyclic $\Z[G]$-module.
Since $r(U')$ is a quotient of $U'$, it is isomorphic, as a $\Z[G]$-module,
to a direct sum of copies of $W$.
If $\llag b'_y \rrag_Q$ were not isomorphic to $W$ as a $\Z[G]$-module,
it would follow that the decomposition $N = r(U') \times \llag b'_y \rrag_Q$
would be the isotypical decomposition of $N$.
Since $\llag b'_y \rrag_Q$ is a simple module, it is cyclic,
and it would follow from Lemma~\ref{lem:IsotypicalDecomposition}
that $N$ is a cyclic $\Z[G]$-module as well.
This contradicts the claim above.
We infer that $\llag b'_y \rrag_Q$ is abelian of exponent $p$,
and isomorphic to $W$ as an $\F_p[G]$-module.
\par 

For each $x \in F \smallsetminus \{y\}$,
the image $r(b_x)$ is contained in a simple $\F_p[G]$-submodule of $N$ contained in $r(U')$.
Since $r(U') = U'/B$ is a direct sum of $j \le m$ copies of $W$,
we deduce from Lemma~\ref{lem:counting}
that $r(U')$ contains $q^{j-1} + \dots + q +1$ simple submodules.
Since $N$ is a direct sum of $j+1$ copies of $W$,
it contains $q^j + q^{j-1} + \dots + q +1$ simple submodules.
Since $q^j \ge 2$, we deduce that $N$ contains a simple $\F_p[G]$-submodule $C$
which is neither contained in $r(U')$ nor equal to $\llag b'_y \rrag_Q$.
The quotient $N/C$ is a direct sum of at most $j \le m$ copies of $W$,
and thus is cyclic by Lemma~\ref{lem:Isotypical}.
Therefore $Q/C$ has an irreducible unitary representation whose restriction to $N/C$ is faithful,
by Lemma~\ref{lem:G-faithful-normal-subgroup} and Proposition~\ref{lem14BEHA}.
By construction, every element of $r(F' \smallsetminus \{b_y\}) \cup \{b'_y\}$
has a non-trivial image in $N/C$.
We conclude that the set $r(F' \smallsetminus \{b_y\}) \cup \{b'_y\}$ is irreducibly faithful in $Q$.
Therefore, as in the proof of the claim above,
we deduce that the set $r(F)$ is also irreducibly faithful in $Q$.
In particular $F$ is irreducibly faithful in $G$.
This final contradiction finishes the proof.
\end{proof}

\subsection{End of proof of Theorem~\ref{thm:CharP(n)}
and proof of Corollary~\ref{cor:P(n)-For-All-n}}

\begin{proof}[Proof of \ref{1DEthm:CharP(n)} $\Rightarrow$ \ref{2DEthm:CharP(n)} in Theorem~\ref{thm:CharP(n)}]
Let $n$ be a positive integer and $G$ a countable group for which
\ref{1DEthm:CharP(n)} of Theorem~\ref{thm:CharP(n)} holds,
i.e., a group which does not have Property $P(n)$.
Upon replacing $n$ by a smaller integer, we may assume that $G$ has Property $P(n-1)$.
(Recall that Property $P(0)$ holds for any group.)
\par

Let $F \subset G$ be an irreducibly unfaithful subset of size $n$.
We invoke Theorem~\ref{thm:UnfaithfulSetSize-n}.
This shows that $U = \llag F \rrag_G$ is a finite normal subgroup
which is an elementary abelian $p$-group,
and which is isomorphic to a direct sum of $\ell+1$ copies of a simple $\F_p[G]$-module $W$
of dimension $m$ over $\ko = \mathcal L_{\F_p[G]}(W)$, where $\ell \ge m$.
In particular $U$ has a submodule $V$
which is isomorphic to a direct sum of $m+1$ copies of $W$.
This proves that \ref{2DEthm:CharP(n)} holds.
\end{proof}

\begin{proof}[Proof of Corollary~\ref{cor:P(n)-For-All-n}]
That \ref{iDEcor:P(n)-For-All-n} implies \ref{iiDEcor:P(n)-For-All-n} is clear.
That \ref{iiDEcor:P(n)-For-All-n} implies \ref{iiiDEcor:P(n)-For-All-n}
follows from Theorem~\ref{thm:CharP(n)}. 
\par

Assume that \ref{iiiDEcor:P(n)-For-All-n} holds.
If $\MA (G) = \{e\}$, then \ref{iDEcor:P(n)-For-All-n} holds by Theorem~\ref{thm:Gasch}.
If not, let $A$ be a non-trivial finite abelian normal subgroup of $G$ contained in the mini-socle.
Let $p$ be a prime dividing $\vert A \vert$; let $A_p$ be the $p$-Sylow subgroup of $A$.
Then $A_p$ is a finite $\F_p[G]$-module,
which is semi-simple because~$A$, hence also $A_p$, is generated by mini-feet of $G$.
Since (iii) holds, 
$A_p$ is a finite simple $\F_p[G]$-module.
By Lemma~\ref{lem:cyclic1conjclass} and Proposition~\ref{prop:cyclicsemi-simple},
$A_p$ is generated by a single conjugacy class.
Since that holds for all $p$ dividing $\vert A \vert$,
it follows that $A$ is generated by a single conjugacy class
(see Lemmas~\ref{lem:cyclic1conjclass} and~\ref{lem:IsotypicalDecomposition}).
Therefore $G$ is irreducibly faithful by Theorem~\ref{thm:Gasch}.
Thus \ref{iDEcor:P(n)-For-All-n} holds. 
\end{proof}

\section{Irreducibly injective sets}
\label{SectionQ(n)}

\subsection{Property \texorpdfstring{$Q(n)$}{Q(n)}}

Recall from Subsection \ref{IntroQ(n)} that
a subset $F$ of a group $G$ is called \textbf{irreducibly injective}
if $G$ has an irreducible unitary representation $\pi$
such that the restriction $\pi \vert_F$ is injective.
We say that $G$ has Property $Q(n)$
if every subset of $G$ of size~$\le n$ is irreducibly injective. 
\par

As mentioned earlier, the fact that every group has $P(1)$
is a classical result of Gelfand--Raikov.
That every group has $Q(1)$ is a trivial fact.
\par

The goal of this section is to compare properties $P(m)$ and $Q(n)$.
For a group $G$ written multiplicatively and for a subset $F$ of $G$, we define
$$
FF^{-1} = \{z \in G \mid z = xy^{-1} \hskip.2cm \text{for some} \hskip.2cm x, y \in F \} .
$$
(When $G$ is abelian and written additively, this is the same as
the subset $F-F$ defined in Section \ref{Section:cyclic}.)
To a subset $F$ of $G$, we associate a subset $\binom{F}{2}$ of $G \smallsetminus\{e\}$
defined as follows.
Let $F^2_{\ne}$ be a subset of $F \times F$
consisting of exactly one of each $(x, y)$, $(y,x)$, for $x,y \in F$ with $x \ne y$.
Then
$$
\binom{F}{2} = \{z \in G \smallsetminus \{e\} \mid z = xy^{-1}
\hskip.2cm \text{for some} \hskip.2cm
(x,y) \in F^2_{\ne} \} .
$$
In particular, if $F$ is a singleton, then $\binom{F}{2}$ is empty;
if $F$ is finite of some size $n \ge 2$, then $\vert \binom{F}{2} \vert \le \binom{n}{2}$.
Note that $\binom{F}{2}$ involves an arbitrary choice (its dependence on $F$ is not canonical),
even though it is not apparent in the notation.
\par

The following lemma records the most straightforward implications
between Properties $P$ and $Q$. 

\begin{lem}
\label{lem:BasicPQ}
Let $G$ be a group and $n$ a positive integer.
\begin{enumerate}[label=(\roman*)]
\item \label{iDElem:BasicPQ}
Let $F$ be a finite subset of $G$ of size $n$;
let $E$ be a finite subset of the form~$\binom{F}{2}$.
Then $F$ is irreducibly injective if and only if $E$ is irreducibly faithful.
\item \label{iiDElem:BasicPQ}
If $G$ has $P\left( \binom{n}{2} \right)$, then $G$ has $Q(n)$. In particular $G$ has $Q(2)$. 
\item \label{iiiDElem:BasicPQ}
If $G$ has $Q(n+1)$, then $G$ has $P(n)$.
\end{enumerate}
\end{lem}

\begin{proof}
Claim \ref{iDElem:BasicPQ} follows from the definitions.
\par

For \ref{iiDElem:BasicPQ},
let $F \subset G$ be a subset of size at most $n$.
Let $E \subset G \smallsetminus \{e\}$ be a subset of the form $\binom{F}{2}$.
Since $G$ has $P\left( \binom{n}{2} \right)$, and as $\vert E \vert \le \binom{n}{2}$,
there exists an irreducible unitary representation $\pi$ of $G$
such that $\pi(z) \ne \id$ for all $z \in E$.
It follows that $\pi(xy^{-1}) \ne \id$ for all $(x,y) \in F^2_{\ne}$,
i.e., $\pi(x) \ne \pi(y)$ for all $(x,y) \in F^2$ with $x \ne y$.
Hence $G$ has $Q(n)$.
Applying this fact to $n=2$, and recalling that every group has $P(1)$,
we deduce that every group has $Q(2)$. 
\par

For \ref{iiiDElem:BasicPQ},
let $F \subset G$ be a subset of size at most $n$.
Since $G$ has $Q(n+1)$,
the set $F \cup \{e\}$ is irreducibly injective.
\end{proof}

Claim \ref{iiiDElem:BasicPQ}
will be strengthened in Proposition \ref{prop:Q(n)=>P(n)}.

\subsection{\texorpdfstring{$Q(n)$}{Q(n)} implies \texorpdfstring{$P(n)$}{P(n)}}

Using Theorem~\ref{thm:CharP(n)},
we obtain for countable groups the following small improvement
of Lemma~\ref{lem:BasicPQ}\ref{iiiDElem:BasicPQ}.

\begin{prop}
\label{prop:Q(n)=>P(n)}
If a countable group has $Q(n)$ for some $n \ge 1$, then it also has $P(n)$. 
\end{prop}

\begin{proof}
Since every group has $Q(1)$ and $P(1)$,
we may assume that $n \ge 2$.
Let $G$ be a countable group satisfying $Q(n)$.
By Lemma~\ref{lem:BasicPQ}\ref{iiiDElem:BasicPQ}, the group $G$ has $P(n-1)$.
\par

Suppose for a contradiction that $G$ does not have $P(n)$.
We may then invoke Theorem~\ref{thm:CharP(n)}.
Let $V, m, q$ be as in Theorem~\ref{thm:CharP(n)}\ref{2DEthm:CharP(n)};
in particular $V$ is a finite abelian normal subgroup of $G$ and we have 
$q^m + \dots + q +1 \le n$.
If we had $q^m + \dots + q +1 < n$,
then the other implication of Theorem~\ref{thm:CharP(n)} would imply
that $G$ does not have $P(n-1)$, in contradiction with the previous paragraph.
We conclude that $q^m + \dots + q +1 = n$.
\par

By Lemma~\ref{lem:F-F},
the group $V$ has a subset $F$ of size $q^m + \dots + q +1$ such that
the set $F-F$ contains a non-zero element
of each abelian mini-foot of $G$ contained in $V$.
By Lemma~\ref{firstpartproofMainTheorem},
given an irreducible unitary representation $\pi$ of $G$,
the kernel $\Ker(\pi)$ intersects $V$ non-trivially.
More precisely, $\Ker(\pi)$ contains an abelian mini-foot of $G$ contained in $V$,
and hence a non-zero element of $F-F$.
Therefore $\pi(x) = \pi(y)$ for some $x \neq y \in F$.
This proves that $F$ is not irreducibly injective.
Since $\vert F \vert = n$,
we deduce that $G$ does not have $Q(n)$, a contradiction. 
\end{proof}

\subsection{The constant \texorpdfstring{$\alpha_{(q,m)}$}{alphaetc}}

Theorem \ref{thm:Q(n)small-n}, which is the main result of this section,
depends on technical results for which we introduce the following notation.
Let $q$ be a power of some prime $p$ and $m \ge 1$ be an integer.
Let $G_{(q, m)} = \mathrm{GL}(W) \ltimes V$ be the group defined in Example~\ref{exampleGqm},
whose notation is retained here.
We define 
$$
\alpha_{(q, m)}
$$
as the smallest cardinality of a subset $F \subset V$
such that the difference set $F-F$ contains a non-zero vector
of each of the $q^m + \dots + q +1$ simple $\F_p[G_{(q, m)}]$-submodules of $V$. 
Lemma~\ref{lem:F-F} implies that the inequality
$$
\alpha_{(q, m)} \le q^m + \dots + q +1
$$
holds for all $q$ and $m$.
The following result shows that the constant $\alpha_{(q, m)}$
is somehow independent of the group $G_{(q, m)}$.

\begin{lem}
\label{lem:alpha(q,m)}
Let $G$ be a group. Suppose that there exist
a prime $p$, a positive integer $m$,
a finite simple $\F_p[G]$-module $W$ of dimension $m$ over $\ko = \mathcal L_{\F_p[G]}(W)$,
and a finite normal subgroup $V$ of $G$ which is an elementary abelian $p$-group
and which is isomorphic as $\F_p[G]$-module to the direct sum of $m+1$ copies of $W$.
\par

Then $\alpha_{(q, m)}$ is equal to the smallest cardinality of a subset $F \subset V$
such that $F-F$ contains a non-zero element of each of the simple $\F_p[G]$-submodules of $V$.
\end{lem}

\begin{proof}
The fact that $W$ is a $\ko[G]$-module yields a homomorphism
$G \to \mathrm{GL}_m(\ko) = \mathrm{GL}(W)$.
Set $L = \mathrm{GL}(W)$.
We may view $V$ both as an $\F_p[G]$-module and as an $\F_p[L]$-module.
Moreover, every simple $\F_p[G]$-submodule is also a simple $\F_p[L]$-submodule.
Since the number of simple $\F_p[G]$-submodules equals the number of simple $\F_p[L]$-sub\-modules
by Lemma~\ref{lem:counting},
we infer that every simple $\F_p[L]$-submodule of $V$
is also a simple $\F_p[G]$-submodule of $V$.
In particular, an additive subgroup of $V$ is a simple $\F_p[L]$-submodule
if and only if it is a simple $\F_p[G]$-submodule.
The required assertion follows.
\end{proof}

Clearly, we have 
$$
\binom{\alpha_{(q, m)}} 2 \ge q^m + \dots + q +1.
$$
The following lemma provides the values of $\alpha_{(q, m)}$ for some small $q$ and $m$.g
The proof of the last item was computer-aided. 
We are grateful to Max Horn for having independently checked the result.

\begin{lem}
\label{lem:alpha(q,m)small} 
With the notation $\alpha_{(q, m)}$ defined above, we have:
\begin{enumerate}[label=(\roman*)]
\item
$\alpha_{(2, 1)} = 3$.
\item 
$\alpha_{(3, 1)} = \alpha_{(4, 1)} = \alpha_{(5, 1)} = 4$.
\item
$\alpha_{(7, 1)} = \alpha_{(8, 1)} = \alpha_{(2, 2)} = 5$.
\item
$\alpha_{(9, 1)} = 6$.
\end{enumerate}
\end{lem}

\begin{proof}
Consider as above the group $G_{(q, m)} = \mathrm{GL}(W) \ltimes V$.
Recall that $q = \vert \ko \vert$, $m = \dim_{\ko}(W)$,
and $V = W \oplus \dots \oplus W$ ($m+1$ times).

\vskip.2cm

If $m =1$, the simple $\F_p[\ko^\times]$-submodules of $V$
coincide with the $1$-dimensional subspaces of $V = \ko^2$.
As noticed above, we have $\binom{\alpha_{(q, 1)}} 2 \ge q +1$.
For $q = 2$ [respectively, $3 \le q \le 5$, $7 \le q \le 9$], 
this implies $\alpha_{(2,1)} \ge 3$ 
[respectively $\alpha_{(q,1)} \ge 4$, $\alpha_{(q,1)} \ge 5$].
For $q \in \{2,3,4,5,7,8\}$, to show that this lower bound on $\alpha_{(q,1)}$ is attained,
it suffices to exhibit a subset $F \subset V$ of the corresponding size
such that $F-F$ has the required property. 
One can check that the following sets do the job,
where the elements of the prime field $\F_p$ are denoted by $0, 1, \dots, p-1$.
\par

For $q = 2$, we set $F = \{(0, 0), (1, 0), (0, 1)\}$.
\par

For $q = 3$, we set $F = \{(0,0), (0, 1), (1, 0), (1, 1)\}$.
\par

For $q = 4$, we set $F = \{(0, 0), (1,0), (0, 1), (1, x)\}$,
where $\ko$ has been identified with $\F_2[x] / (x^2+x+1)$.
\par

For $q=5$, we set $F = \{(0, 0), (1,0), (0,1), (3, 4)\}$.
\par

For $q=7$, we set $F = \{ (0, 0), (1,0), (0,1), (2, 3), (5, 2)\}$.
\par

For $q=8$, we set $F = \{ (0, 0), (1,0), (0,1), (1, x), (x^2+x, x^2)\}$,
where $\ko$ has been identified with $\F_2[x] / (x^3 +x +1)$.
\par

For $q = 9$, the situation is different.
We know that $\binom{\alpha_{(9, 1)}} 2 \ge 10$, so that $\alpha_{(9, 1)} \ge 5$.
With the help of a computer, we checked that no subset $F$ in $V$ of size $5$
is such that $F-F$ contains a non-zero vector
of each of the $10$ one-dimensional subspaces of $V$.
On the other hand, one verifies that the set
$$
F= \{ (0, 0), (1,0), (0,1), (0, 2), (0, x), (2, 2x+1) \}
$$
has this property, where $\ko$ has been identified with $\F_3[x] / (x^2-x-1)$.
Thus $\alpha_{(9, 1)} = 6$.

\vskip.2cm

Finally, consider the case of $m=2$ and $q=2$.
Since $\binom{\alpha_{(2, 2)}} 2 \ge 2^2 + 2 +1 = 7$,
we have $\alpha_{(2, 2)} \ge 5$.
Let $a$ be a non-zero vector in~$W$.
One checks that the set
$$
F = \{(0, 0, 0), (a, 0, 0), (0, a, 0), (0, 0, a), (a, a, a)\}
$$
satisfies the required condition, so that $\alpha_{(2, 2)} = 5$. 
\end{proof}

\subsection{\texorpdfstring{$P(\binom{n}{2} -1)$}{}
sometimes implies \texorpdfstring{$Q(n)$}{}}

We are now ready to present the main technical result of this section.
It may be viewed as a supplement to Lemma~\ref{lem:BasicPQ}\ref{iiDElem:BasicPQ}.

\begin{prop}
\label{prop:nonQ(n)}
Let $n$ be an integer, $n \ge 3$.
Let $G$ be a countable group with Property $P\left( \binom{n}{2} -1 \right)$.
Assume that, for all pairs $(q, m)$ consisting of
a prime power $q$ and an integer $m$ such that $q^m + \dots + q +1 = \binom{n}{2}$,
we have $\alpha_{(q, m)} > n$.
\par

Then $G$ has $Q(n)$.
\end{prop}

\begin{proof}
Suppose for a contradiction that $G$ does not have $Q(n)$.
Let $F \subset G$ be a subset of size $\le n$
which is not irreducibly injective in~$G$.
Upon replacing $F$ by $Fx^{-1}$ for some $x \in F$,
we may assume without loss of generality that $F$ contains the neutral element $e$.
\par

Let $E \subset G \smallsetminus \{e\}$ be a subset of the form $\binom{F}{2}$;
recall that $\vert E \vert \le \binom{n}{2}$.
Since $e \in F$, we may choose $E$ in such a way that $E$ contains $F \smallsetminus \{e\}$.
It follows from Lemma \ref{lem:BasicPQ}\ref{iDElem:BasicPQ} that $E$ is irreducibly unfaithful.
Since $G$ has $P\left( \binom{n}{2} -1 \right)$ by hypothesis,
we deduce that $\vert E \vert = \binom{n}{2}$.
Set $U = \llag E \rrag_G$.
Since $F \smallsetminus \{e\} \subset E$, we have $F \subset U$.
\par

We invoke Theorem~\ref{thm:UnfaithfulSetSize-n} and use its notation,
except for $F$ there being $E$ here.
In particular, there exist a prime $p$ and a simple $\F_p[G]$-module $W$
such that $U$ is isomorphic as an $\F_p[G]$-module
to the direct sum of $\ell+1$ copies of $W$ for some $\ell \ge m$.
By Theorem~\ref{thm:UnfaithfulSetSize-n}\ref{iiiDEthm:UnfaithfulSetSize-n},
we have $q^m + \dots + q +1 = \binom{n}{2}$.
Set $V = \bigoplus_0^m W$. 
\par

We next claim that there exists a surjective map of $\F_p[G]$-modules
$r \hskip.1cm \colon U \twoheadrightarrow V$
whose restriction to $F$ is injective.
If $\ell = m$, then $U = V$ and $r$ can be defined as the identity map.
If $\ell > m$, we proceed by induction on $\ell - m$.
Lemma~\ref{lem:counting} ensures
that the number of simple $\F_p[G]$-submodules of $U$
is strictly larger than $\binom{n}{2}$.
Since $\binom{n}{2} = \vert E \vert$,
there exists a simple $\F_p[G]$-submodule $U_0$ of $U$
such that $U_0 \cap E = \{0\}$.
If $r_0 \hskip.1cm \colon U \twoheadrightarrow U/U_0$
denotes the quotient map,
we have $\Ker(r_0) \cap E = \{0\}$,
and it follows that the restriction of $r_0$ to $F$ is injective.
Since $U/U_0$ is isomorphic to a direct sum of~$\ell$ copies of $W$,
the induction hypothesis guarantees the existence
of a surjective map of $\F_p[G]$-modules
$r_1 \hskip.1cm \colon U/U_0 \twoheadrightarrow V$
whose restriction to $r_0(F)$ is injective.
The map $r = r_1 \circ r_0 \hskip.1cm \colon U \twoheadrightarrow V$
satisfies the required property.
This proves the claim.
\par

Set $E' = r(E)$, $F' = r(F)$ and $K = \Ker(r)$.
Since $K$ is an $\F_p[G]$-submodule of $U$, we may view it as a normal subgroup of $G$.
We view $E'$, $F'$ and $V$ as subsets of the quotient group $G' = G/K$;
observe that $E' \subset G' \smallsetminus \{e\}$ is of the form $\binom{F'}{2}$.
Since $F$ is not irreducibly injective in $G$, it follows that $F'$ is not irreducibly injective in $G'$.
Hence $E'$ is not irreducibly faithful in~$G'$.
Therefore $E'$ contains a non-zero element in each of the simple submodules of $V$,
by Lemma~\ref{firstpartproofMainTheorem}.
Recalling that $E' \subset F' - F'$, we deduce from Lemma~\ref{lem:alpha(q,m)}
that $\alpha_{(q, m)} \le \vert F' \vert$.
Since $\vert F' \vert = \vert F \vert = n$,
this contradicts the hypothesis that $\alpha_{(q, m)} > n$.
\end{proof}

\begin{rem}
As mentioned in Section~\ref{subsec:1.a}, the Goormaghtigh Conjecture
predicts that for every integer $\ell$,
there exists at most one prime power $q$ and one positive integer $m$
such that $q^m + \dots + q + 1 = \ell$, except for $\ell = 31$.
Since $31$ is not of the form $n \choose 2$,
that conjecture predicts that the condition from Proposition~\ref{prop:nonQ(n)}
needs to be checked for at most one value of $q$ and $m$, once the integer $n$ is fixed.
\end{rem}

\begin{thm}
\label{thm:Q(n)small-n}
Let $G$ be a group. Then $G$ has Properties $P(2)$ and $Q(2)$.
Suppose moreover that $G$ is countable; then:
\begin{enumerate}[label=(\roman*)]
\item \label{iDEthm:Q(n)small-n}
$G$ has $Q(3)$ if and only if $G$ has $P(3)$;
\item \label{iiDEthm:Q(n)small-n}
$G$ has $Q(4)$ if and only if $G$ has $P(6)$. 
\item \label{iiiDEthm:Q(n)small-n}
$G$ has $Q(5)$ if and only if $G$ has $P(9)$. 
\end{enumerate}
\end{thm}

\begin{proof}
By Lemma~\ref{lem:BasicPQ}\ref{iiDElem:BasicPQ},
Property $P( \binom{n}{2} )$ implies $Q(n)$.
For $n=3$, and $4$,
this yields $P(3) \Rightarrow Q(3)$ and $P(6) \Rightarrow Q(4)$.
By Proposition~\ref{prop:Q(n)=>P(n)}, we have $Q(3) \Rightarrow P(3)$.
\par

Among other things, this proves~\ref{iDEthm:Q(n)small-n}.

\vskip.2cm

Let now $G$ be a countable group that does not satisfy $P(6)$.
To show \ref{iiDEthm:Q(n)small-n}, it remains to show
that $G$ does not have $Q(4)$.
We may assume that $G$ has $Q(3)$,
since otherwise we are already done.
Hence, $G$ has $P(3)$ by~\ref{iDEthm:Q(n)small-n}.
Let $n$ be the least integer such that $G$ does not have $P(n)$.
Hence $n$ is one of $4$, $5$, or $6$. 
\par

If $n = 4$, we deduce from Theorem~\ref{thm:CharP(n)}
that $G$ contains a normal subgroup $V$ isomorphic to $\F_3 \oplus \F_3$,
on which the $G$-action is by scalar multiplication.
Let $\pi$ be an irreducible unitary representation of $G$.
Set $Q = G/\Ker(\pi)$
and let $r \hskip.1cm \colon G \twoheadrightarrow Q$ be the canonical projection. 
By Proposition~\ref{structureminisocle}\ref{7DEstructureminisocle},
the subgroup $r(V) \le Q$ is generated by abelian mini-feet of $Q$,
and it is an elementary abelian $3$-group.
Suppose that $r(V)$ were isomorphic to $V$;
note that $Q$ would act on $V$ by scalar multiplication;
since $Q$ is irreducibly faithful, hence has Property $P(4)$,
this would contradict Theorem~\ref{thm:CharP(n)}.
Hence the restriction of $r$ to $V$ cannot be faithful.
(Note moreover that, for each of the simple $\F_3[G]$-modules $W$
contained in $V$, the restriction to $W$ of the projection $r$
is either injective or the zero map.)
Therefore $\Ker(r) = \Ker(\pi)$ contains
at least one of the $4$ cyclic subgroups of order~$3$ of $V$.
Lemma~\ref{lem:alpha(q,m)small} yields a subset $F$ of $V$ of size $4$
such that $F - F$ contains a non-trivial element
of each fo the $4$ cyclic subgroups of order~$3$ of $V$.
Therefore $\pi(a) = \pi(b)$ for some $a, b$ distinct in $F$.
This shows that $G$ does not have Property $Q(4)$. 
\par

If $n = 5$ and $n = 6$, similar arguments using
Lemmas \ref{lem:alpha(q,m)} and \ref{lem:alpha(q,m)small}
apply, each time with $\vert F \vert = 4$.
This confirms that \ref{iiDEthm:Q(n)small-n} holds.

\vskip.2cm 

Arguing similarly using Theorem~\ref{thm:CharP(n)} and Lemma~\ref{lem:alpha(q,m)small},
we see that $Q(5)$ implies $P(9)$.
Conversely, invoking Proposition~\ref{prop:nonQ(n)} for $n = 5$,
we deduce that $P(9)$ implies $Q(5)$
since $\alpha_{(9, 1)} = 6$ by Lemma~\ref{lem:alpha(q,m)small}.
This proves \ref{iiiDEthm:Q(n)small-n}.
\end{proof}

\subsection{From \texorpdfstring{$Q(n)$}{} to additive combinatorics}
\label{AdditiveComb}

Theorem~\ref{thm:Q(n)small-n} suggests the following question.

\begin{ques}
\label{ques:Q(n)}
Can we characterize Property $Q(n)$ by an algebraic property of $G$,
in the same vein as in Theorem~\ref{thm:CharP(n)}? 
\par

In particular, is it true that, for each $n \ge 1$, there exists an integer $f(n) \ge 1$
such that a countable group $G$ has Property $Q(n)$
if and only if it has Property~$P(f(n))$ ?
\end{ques}

The proof of Theorem~\ref{thm:Q(n)small-n} 
suggests that an answer to Question~\ref{ques:Q(n)} might require
to compute the numbers $\alpha_{(q, m)}$ for all $(q, m)$.
This is confirmed by the following observation.

\begin{obs}
\label{obs:n_p}
The group $G_{(q, m)}$ of Example~\ref{exampleGqm} has Property $Q(\alpha_{(q, m)}-1)$,
but not $Q(\alpha_{(q, m)})$.
\end{obs}

\begin{proof}
That $G = G_{(q, m)}$ does not have $Q(\alpha_{(q, m)})$
follows from the definition and from Lemma~\ref{lem:BasicPQ},
in view of Theorem~\ref{thm:CharP(n)}.
\par

In order to show that $G_{(q, m)}$ has $Q(\alpha_{(q, m)}-1)$,
we fix a subset $F$ of $G$ such that $\vert F \vert < \alpha_{(q, m)}$.
We shall prove that $FF^{-1}$ is irreducibly faithful.
This implies that $F$ is irreducibly injective, as required.
Modules below refer to the ring $\F_p[G]$.
\par

Notice that $FF^{-1}$ remains unchanged when $F$ is replaced by a translate $Fg$,
for some $g \in G$.
Without loss of generality we may thus assume that $F$ contains~$e$.
In particular $F \subseteq FF^{-1}$. 
\par

Let $\{g_1, \dots , g_k\} \subset G$ be a set of minimal cardinality
such that $F \subset \bigcup_{i=1}^k V g_i$.
For each $i$, set $F_i = F \cap V g_i$.
Notice that if $x \in F_i$ and $y \in F_j$ with $i \neq j$,
then $xy^{-1} \not \in V$ because $g_ig_j^{-1} \notin V$.
Therefore the intersection $FF^{-1} \cap V$ coincides with $\bigcup_{i=1}^k F_i F_i^{-1}$.
For each $i$, we set $F'_i = F_i g_i^{-1}$, and set $F' = \bigcup_{i=1}^k F'_i$.
Hence
$$
\vert F' \vert \le \vert F \vert , \hskip.2cm
F' \subseteq V 
\hskip.2cm \text{and} \hskip.2cm
F'(F')^{-1} \supseteq FF^{-1} \cap V .
\leqno{(\sharp)}
$$
\par

We next observe that, if $W$ is any simple submodule of $V$,
then the quotient group $G/W$ is irreducibly faithful.
This follows from Theorem~\ref{thm:CharP(n)} and Corollary~\ref{cor:P(n)-For-All-n}
(using a similar argument as in the discussion of Example~\ref{exampleGqm}). 
Therefore, if $FF^{-1}$ were not irreducibly faithful,
then it would contain a non-zero element of each simple submodule of $V$.
By ($\sharp$), $F'(F')^{-1}$ would also contain
a non-zero element of each simple submodule of $V$.
This would contradict the sequence of inequalities
$\vert F' \vert \le \vert F \vert < \alpha_{(q,m)}$.
It follows that $F F^{-1}$ is irreducibly faithful, and this ends the proof.
\end{proof}

In particular, answering Question~\ref{ques:Q(n)} for $C_p \times C_p = G_{(p, 1)}$
amounts to compute $\alpha_{(p, 1)}$. 
This happens to be an open problem in additive combinatorics,
see Question~5.2 in \cite{CrSL--07}.
As pointed out in this reference,
the value of $\alpha_{(p, 1)} = n_p$ can be estimated as follows.
On the one hand, since
$$
\frac {n_p^2} 2 > \binom{n_p}{2} \ge p+1 > p,
$$
we have $n_p > \sqrt{2p}$.
On the other hand, using Theorems 1.2 and 2.1 from \cite{FiJa--00},
we obtain the upper bound
$$
n_p \le 2 \lceil \sqrt p \rceil +1.
$$
However, determining the exact value of $n_p$ remains an open problem.
We are grateful to Ben Green for point out the reference \cite{CrSL--07}
and for discussing it with us.

\appendix

\section{A finite group all of whose irreducible representations have non-abelian kernels}
\label{Section:appendixA}

We know from Proposition~\ref{prop:Kernel}
that every countable group $G$ has an irreducible unitary representation
whose kernel is contained in $\Fsol (G)$,
and we have cited in Remark \ref{BrolineGarrison} the result according to which
every finite group has an irreducible representation with nilpotent kernel.
Short of having found in the literature appropriate references
for groups without irreducible representations having abelian kernels,
we indicate here an example, long known to experts.

\vskip.2cm

Let $D_8$ denote the dihedral group of order~$8$.
The centre of $D_8$ is cyclic of order~$2$.
For $i = 1, 2, 3$, let $H_i$ be a group isomorphic to $D_8$,
and let $z_i$ be the non-trivial element of the centre of $H_i$.
We set
$$
G = (H_1 \times H_2 \times H_3) / \langle z_1 z_2 z_3 \rangle.
$$
Thus $G$ is a nilpotent group of order $2^{8} = 256$.
Its centre $Z(G)$ is isomorphic to $C_2 \times C_2$.
The socle of $G$ coincides with its centre,
and $\Fsol (G)$ is the group $G$ itself (see Example \ref{exe:socletrunk}(6)).

\begin{prop}
\label{prop:AbelianNormal}
For every abelian normal subgroup $N$ in $G$, the centre of the quotient 
$G/N$ is not cyclic.
\end{prop}

\begin{proof}
The natural homomorphism $H_1 \times H_2 \times H_3 \to G$
induces an embedding $H_i \to G$ for each $i$.
We identify $H_i$ with its image in $G$.
In particular we view $z_1, z_2, z_3$ as elements of $G$.
The centre of $G$ is $Z(G) = \{e, z_1, z_2, z_3\}$.
\par

We assume for a contradiction that $N$ is an abelian normal subgroup of 
$G$ such that $G/N$ has a cyclic centre.
Since the centre of $G$ is not cyclic, we have $N \neq \{e\}$.
Let $r \hskip.1cm \colon G \twoheadrightarrow G/N$ be the canonical projection.
\par

Let $i \in \{1, 2, 3\}$.
Since $N$ is abelian and $H_i$ is not,
$r(H_i ) \cong H_i / H_i \cap N$ is non-trivial.
In particular $r(H_i )$ has a non-trivial centre.
We may thus choose an element $h_i \in H_i$
such that $r(h_i)$ is a non-trivial element of the centre $Z(r(H_i))$.
Since $H_1, H_2$ and $H_3$ commute pairwise in $G$,
and since $G$ is generated by these subgroups,
we have $Z(H_i) \le Z(G)$ and $Z(r(H_i)) \le Z(r(G))$.
In particular $r(h_i) \in Z(G/N)$.
\par

Since $G$ is a $2$-group, every non-trivial normal subgroup
has a non-trivial intersection with the centre $Z(G)$.
Thus there exists $j \in \{1, 2, 3\}$ such that $z_j \in N$.
Since $Z(G/N)$ is cyclic and $Z(r(H_j)) \le Z(G/N)$, it follows that $Z(r(H_j))$ is cyclic.
Since $N \cap H_j$ is a non-trivial normal subgroup of $H_j \cong D_8$,
the quotient $ r(H_j) \cong H_j / (N \cap H_j)$ is abelian.
Therefore, it coincides with its centre, hence it is cyclic.
The only abelian normal subgroups of $H_j \cong D_8$ affording a cyclic quotient group
are its subgroups of index~$2$.
Thus $r(H_j) \cong H_j/ (N \cap H_j)$ is of order~$2$.
In particular $N \cap H_j$ is a maximal subgroup of $H_j$, and $r(h_j)$ is of order~$2$.
Moreover we have
$H_j = \langle h_j \rangle (N \cap H_j)$ since $r(h_j) \neq e$
and hence $h_j \not \in N \cap H_j$.
\par

Let now $i \in \{1, 2, 3\}$ such that $i \neq j$.
We know that $Z(G/N)$ is cyclic,
and that $r(h_i)$ and $r(h_j)$ are two non-trivial elements in $Z(G/N)$.
Since moreover $r(h_j)$ is of order~$2$, we infer that 
$r(h_i)^k r(h_j) = e$ for some integer $k$.
In other words $h_i^k h_j \in N$.
Since $N$ is abelian, it follows that $h_i^k h_j$ commutes with $N \cap H_j$.
Moreover $h_i$ commutes with $H_j$, hence $h_i^k$ commutes with $N \cap H_j$.
It follows that $h_j$ commutes with $N \cap H_j$. 
Since $N$ is abelian and since $H_j = \langle h_j \rangle (N \cap H_j)$, 
it follows that $H_j \cong D_8$ is abelian, which is absurd.
\end{proof}

By Schur's Lemma, the image $\pi(G)$ of $G$
under any irreducible representation $\pi$ has cyclic centre.
It follows from Proposition~\ref{prop:AbelianNormal}
that the kernel of $\pi$ cannot be abelian. Thus we obtain:

\begin{cor}
\label{cor:Appendix}
Every irreducible representation of $G$ has a non-abelian kernel.
\end{cor}

Here is another proof of Corollary \ref{cor:Appendix}.
Let $z$ denote the non-trivial element of the centre of $D_8$.
The group $D_8$ has $4$ irreducible representations of degree $1$,
of which the kernels contain $z$,
and one irreducible representation $\pi$ of degree $2$,
such that $\pi(z) = -\id$.
Consequently, the group $H_1 \times H_2 \times H_3$ has
\begin{enumerate}
\item[---]
$64$ irreducible representations of degree $1$, 
\item[---]
48 irreducible representations of degree $2$, 
\item[---]
$12$ irreducible representations of degree $4$, 
\item[---]
$1$ irreducible representation of degree $8$, 
\end{enumerate}
and the irreducible representations having kernels containing $z_1z_2z_3$
are precisely those of dimensions $1$ and $4$.
It follows that $G$ has $64$ irreducible representations of degree $1$,
none of them with abelian kernel,
and $12$ irreducible representations of degree $3$,
each having a kernel isomorphic to $D_8$.

\section{On the collection of kernels of irreducible unitary representations}
\label{Section:appendixB}

Given a group $G$, we denote by $\Sub(G)$ the set of all subgroups of $G$,
endowed with the \textbf{Chabauty topology}.
In this appendix, we collect some observations
concerning the subspace $\mathcal K_G$ of $\Sub(G)$
of all kernels of irreducible unitary representations of $G$,
for comparison with the situation in the particular case of finite groups.

\vskip.2cm

We begin by a short reminder on the space $\Sub(G)$.
In a group $G$, every subset can be identified, in a canonical way,
with a function $f \hskip.1cm \colon G \to \{0, 1\}$.
The set $\{0, 1\}^G$ of all such functions,
endowed with the topology of pointwise convergence,
is compact by Tychonoff's theorem.
The set $\Sub(G)$ is closed in $\{0, 1\}^G$,
because being a subgroup is a pointwise condition.
Endowed with the induced topology,
the space $\Sub(G)$ is called the \textbf{Chabauty space} of subgroups of $G$.
\par

By definition of the topology, a sequence $(K_n)_{n \ge 1}$ in $\Sub(G)$
converges to a subgroup $K \le G$
if and only if, for every finite subset $F \subset G$,
the intersection $K_n \cap F$ is equal to $K \cap F$ for all sufficiently large $n$. 
When $G$ is countable, it suffices to check the latter condition on the finite sets $F$
belonging to some ascending chain $F_1 \subset F_2 \subset \dots$
of finite subsets of $G$ such that $\bigcup_{m \ge 1} F_m = G$.

\begin{exe}
\label{exKGnotclosed}
Our first observation is that the subspace $\mathcal K_G$
need not be a closed subset of the Chabauty space $\Sub(G)$.
\par
 
For this, consider the group
$$
G = \langle x, y_n \mid n \ge 1 \rangle
$$
of Example~\ref{exe:socletrunk}(7).
Recall that $G$ is a subgroup of $P = \prod_{n \ge 1} H_n$, where
$$
H_n = \langle x_n, y_n, z_n \mid 
x_n^3, \hskip.1cm y_n^3, \hskip.1cm
[x_n, y_n]z_n^{-1}, \hskip.1cm 
[x_n, z_n], \hskip.1cm [y_n, z_n] \rangle
$$
for each $n \ge 1$, and $x = (x_n)_{n \ge 1}$.
For $m \ge 1$, define the subgroup
$$
K_m = \langle
y_j^{-1}z_{m+1}, \hskip.1cm y_{m+1}^{-1} y_k, \hskip.1cm z_j, \hskip.1cm z_{m+1}^{-1} z_k
\mid
1 \le j \le m, \hskip.1cm k \ge m+2 \rangle
$$
of $G$. In Lemma \ref{lemmaforB1} and just after, we will show that:
\begin{enumerate}
\item[---]
for all $m \ge 1$, the group $K_m$ is normal in $G$ and is in $\mathcal K_G$;
\item[---]
the sequence $(K_m)_{m \ge 1}$ converges in $\Sub(G)$
to a normal subgroup $K$ of $G$ which is not in $\mathcal K_G$.
\end{enumerate}
\end{exe}

\begin{lem}
\label{lemmaforB1}
Let the notation be as just above. Let $m \ge 1$.
\begin{enumerate}[label=(\roman*)]
\item \label{iDElemmaforB1}
The group $K_m$ is normal in $G$.
\item \label{iiDElemmaforB1}
We have $G = K_m \rtimes \langle x, y_{m+1}, z_{m+1} \rangle$.
In particular $G/K_m$ is of order $27$, it is isomorphic to $H_1$,
which is a Heisenberg group over $\F_3$.
\item \label{iiiDElemmaforB1}
The assignments
$$
\left\{
\begin{array}{rcll}
x & \mapsto & x_1 \\
y_j & \mapsto & z_1 & \text{for all } j \le m \\
y_j & \mapsto & y_1 & \text{for all } j \ge m+1 \\
z_j & \mapsto & e & \text{for all } j \le m \\
z_j & \mapsto & z_1 & \text{for all } j \ge m+1
\end{array}
\right.
$$
extend to a uniquely defined group homomorphism 
$\rho_m \hskip.1cm \colon G \twoheadrightarrow H_1$
which is surjective.
\end{enumerate}
\end{lem}

\begin{proof}
Recall that we have defined in Example~\ref{exe:socletrunk}(7) the subgroup
$$
A = \langle y_j, z_k \mid j, k \ge 1 \rangle .
$$
It is abelian, normal and of index $3$ in $G$.
Observe that $K_m \le A$. 

\vskip.2cm

\ref{iDElemmaforB1}
The generators defining $K_m$ are either central in $G$,
or of the form $y_j^{-1}z_{m+1}$ for $j \le m$,
or of the form $y_{m+1}^{-1} y_k$ for $k \ge m+2$.
Every conjugate of $y_j^{-1}z_{m+1}$ in $G$
belongs to the coset $y_j^{-1} z_{m+1} \langle z_j \rangle$,
which is entirely contained in $K_m$ since $j \le m$.
Similarly, the conjugacy class of $y_{m+1}^{-1} y_k$ in $G$ is of size at most~$3$
(because the centralizer $C_G(y_{m+1}^{-1} y_k)$ contains $A$,
which is of index~$3$ in $G$),
and is contained in the coset $y_{m+1}^{-1} y_k \langle z_{m+1}^{-1} z_k \rangle$,
which is entirely contained in $K_m$ as well since $k \ge m+2$.
Thus $K_m$ contains the conjugacy class of each of its generators,
and therefore $K_m$ is normal in $G$.

\vskip.2cm

\ref{iiDElemmaforB1}
Let $D = \langle x, y_{m+1}, z_{m+1} \rangle$.
The natural projection $P \twoheadrightarrow H_{m+1}$
restricts to an isomorphism $D \overset{\cong}{\longrightarrow} H_{m+1}$.
In particular $D \cong H_{m+1} \cong H_1$. We must show that $G = K_m \rtimes D$.
\par

We have $A = \langle y_{m+1}, z_{m+1}\rangle K_m$. 
Viewing $A$ as a vector space over $\F_3$
with basis $\{y_1, z_1, y_2, z_2, \dots \}$,
we obtain a direct product decomposition $A = \langle y_{m+1}, z_{m+1} \rangle \times K_m$. 
Since $A \cap D = \langle y_{m+1}, z_{m+1} \rangle$
because $x \not \in A$ as observed in Example~\ref{exe:socletrunk}(7),
we deduce that $K_m \cap D = \{e\}$.
Since $K_m D$ contains $A$ as a proper subgroup, and since $[G : A]= 3$,
we infer that $G = K_m D$.
This confirms that $G$ is the semi-direct product
$K_m \rtimes \langle x, y_{m+1}, z_{m+1} \rangle$.

\vskip.2cm

\ref{iiiDElemmaforB1}
Observe that, modulo $K_m$, we have
$$
\left\{
\begin{array}{rcllll}
& y_j &\equiv &z_{m+1} &\pmod{K_m} &\hskip.2cm \text{for all} \hskip.2cm j \le m \\
& y_j &\equiv &y_{m+1} &\pmod{K_m} &\hskip.2cm \text{for all} \hskip.2cm j \ge m+1 \\
& z_j &\equiv &e &\pmod{K_m} &\hskip.2cm \text{for all} \hskip.2cm j \le m \\
& z_j &\equiv &z_{m+1} &\pmod{K_m} &\hskip.2cm \text{for all} \hskip.2cm j \ge m+1
\end{array}
\right.
$$
The homomorphism $\rho_m$ is the composition
of the canonical projection $G \twoheadrightarrow G/K_m$
with the isomorphism $G/K_m \to H_1$
mapping $xK_m $ to $x_1$, and $y_{m+1}K_m$ to $y_1$, and $z_{m+1}K_m$ to $z_1$. 
\end{proof}

\begin{proof}[End of proof of the claims of Example \ref{exKGnotclosed}]
The group $H_1$ is irreducibly faithful; 
this can be seen on the character table of the group
(see, for example, Page 216 in \cite{Kowa--14});
alternatively, it follows from Theorem~\ref{thm:Gasch}
because $H_1$ is a nilpotent group with cyclic centre.
Therefore $K_m = \Ker(\rho_m)$ is in $\mathcal K_G$ for all $m \ge 1$.
\par

Let now $\rho \hskip.1cm \colon G \to H_1$
be the group homomorphism defined by the assignments
$$
\left\{
\begin{array}{rcll}
x & \mapsto & x_1 \\
y_j & \mapsto & z_1 & \text{for all } j \ge 1 \\
z_j & \mapsto & e & \text{for all } j \ge 1.
\end{array}
\right.
$$
Thus $\rho(G) = \langle x_1, z_1\rangle \cong C_3 \times C_3$.
In particular $K := \Ker(\rho) \not \in \mathcal K_G$.
\par 

It remains to observe that the sequence $(\rho_m)_{m \ge 1}$ converges,
in the topology of pointwise convergence, to the homomorphism $\rho$.
This readily implies that $(K_m)_{m \ge 1}$ converges to $K$
in the Chabauty topology.
Thus $(K_m)_{m \ge 1}$ is a sequence in $\mathcal K_G$
converging to a normal subgroup of $G$ not belonging to $\mathcal K_G$. 
\end{proof}

Despite the fact that it need not be closed in $\Sub (G)$,
the set $\mathcal K_G$ always contains minimal elements. 

\begin{prop}
\label{prop:Chabauty}
Let $G$ be a countable group
and $\mathcal K_G$ be as above
the set of all kernels of irreducible unitary representations of $G$.
\par 

Then, for every descending chain $(K_i)_{i \in I}$ in $\mathcal K_G$,
the intersection $\bigcap_{i \in I} K_i$ belongs to $\mathcal K_G$.
In particular every $K \in \mathcal K_G$ contains a minimal $K_0 \in \mathcal K_G$. 
\end{prop}

\begin{proof}
Let $(K_i)_{i \in I}$ be a descending chain in $\mathcal K_G$.
Set $K = \bigcap_{i \in I} K_i$.
We must show that $G/K$ is irreducibly faithful. 
To that end, we consider a normal subgroup $A$ of $G$ containing $K$,
such that $A/K$ is a finite abelian normal subgroup of $G/K$ contained in $\MA(G/K)$.
\par

We have $K \le A \cap K_i \le A$ for all $i$. 
Since $A/K$ is finite, it follows that, for all sufficiently large $i$,
we have $A \cap K_i = K$, and therefore $A/ (A \cap K_i) = A/K$.
Since $A/K \le \MA(G/K)$,
we have $A/ (A \cap K_i) \le \MA(G/ (A \cap K_i))$ for $i$ large enough.
By Proposition~\ref{structureminisocle}\ref{7DEstructureminisocle}
applied to $G/ (A \cap K_i) \twoheadrightarrow G/K_i$, 
the quotient $A/ (A \cap K_i) \cong AK_i/K_i$ is in $\MA(G/K_i)$.
Since $G/K_i$ is irreducibly faithful by the definition of $\mathcal K_G$,
the normal subgroup $AK_i/K_i $ is generated by a single conjugacy class.
Since the canonical isomorphism $A/ (A \cap K_i) \cong AK_i/K_i$ is $G$-equivariant,
it follows that $A/ (A \cap K_i)$ is generated by a single conjugacy class.
Thus the same holds for $A/K$.
It follows that $G/K$ is irreducibly faithful by Theorem~\ref{thm:Gasch}. 
\par 

The second assertion follows from the first via Zorn's Lemma. 
\end{proof}

\begin{ques}
\label{QuestionBroGar}
In the case of a finite group $G$, 
the theorem of Broline and Garrison quoted in Remark \ref{BrolineGarrison}
ensures that every minimal element $K \in \mathcal K_G$ is nilpotent.
For a countable group $G$, 
we already know from Proposition~\ref{prop:Kernel} that
some element of $\mathcal K_G$ is contained in $\Fsol(G)$. 
\par 

For a countable group $G$, can we describe more precisely
the algebraic structure of the minimal elements of $\mathcal K_G$ ?
Can we always find an element $K \in \mathcal K_G$
contained in the characteristic subgroup $\Fnil (G)$
generated by all finite nilpotent normal subgroups of $G$ ? 
In case $G$ is finite, the subgroup $\Fnil (G)$ coincides with the \textbf{Fitting subgroup},
i.e., the largest nilpotent normal subgroup of $G$.
\end{ques}	

We finish by mentionning that, in the case of infinite groups,
some minimal elements of $\mathcal K_G$ may fail to be nilpotent
and also fail to be contained in the torsion FC-centre $W(G)$;
a fortiori it need not be contained in $\Fnil (G)$,
as defined in Question \ref{QuestionBroGar},
or $\Fsol (G)$, as defined in Lemma \ref{lem:N1Nk}.
The construction of such examples is based on the following.

\begin{lem}
\label{lem:MinimalKernels}
Let $p$ be a prime. For each $n \in \N$,
let $G_n$ be a (possibly infinite) non-trivial nilpotent $p$-group.
Let $K$ be the restricted sum $\prod'_{n \in \N} G_n$,
and set $G = K \times C_p$.
\par 

Then $K$ is a minimal element of $\mathcal K_G$.
\end{lem}

\begin{proof}
Since $G/K \cong C_p$, it is clear that $K$ belongs to $\mathcal K_G$.
Let $K_0 \in \mathcal K_G$ be such that $K_0 \le K$,
and let $r \hskip.1cm \colon G \twoheadrightarrow G/K_0$ denote the canonical projection.
\par

Assume for a contradiction that $r(G_n)$ is non-trivial for some $n \in \N$.
On the one hand, the centre of $r(G_n)$ is an abelian $p$-group,
hence it contains some element of order $p$.
That element commutes with $r(G_m)$ for all $m$, and also with $r(C_p)$.
Thus it belongs to the centre of $r(G) = G/K_0$.
On the other hand, the factor $C_p$ is central in $G$,
hence $r(C_p)$ is also contained in the centre of $r(G)$.
Since $r(G)$ does not contain any group isomorphic to $C_p \times C_p$,
it follows that $r(G_n)$ and $r(C_p)$ have a non-trivial intersection.
Therefore $\Ker(r) = K_0$ contains an element of the form $xy$
with $x \in G_n \smallsetminus \{e\}$ and $y \in C_p \smallsetminus \{e\}$.
This is impossible since $K_0 \le K$.
We have thus proven that $r(G_n) = \{e\}$ for all $n$.
\par

It follows that $K \le K_0$, hence that $K = K_0$.
This confirms that $K$ is minimal in $\mathcal K_G$. 
\end{proof}

\begin{exe}
Consider the group $G$ of Example~\ref{exe:socletrunk}(7),
which is an infinite countable nilpotent $3$-group; denote it now by $G_0$.
For each positive integer $n$, let $G_n$ be a finite nilpotent $3$-group;
assume that the derived length of $G_n$ tends to infinity with $n$.
Let $K$ be the restricted direct sum $\prod'_{n \in \N} G_n$
By Lemma~\ref{lem:MinimalKernels},
the countable group $K \times C_3$ has the following property:
\par
\begin{center}
\emph{the subgroup $K \le K \times C_3$ is a minimal element in $\mathcal K_{K \times C_3}$
which is
\\
neither soluble, nor contained in the torsion FC-centre $W(K \times C_3)$.}
\end{center}
\noindent
Indeed $K$ has an infinite conjugacy class,
because $G_0$ has one,
since (with the notation of Example~\ref{exe:socletrunk}(7))
$y_j^{-1}xy_j = xz_j$ for all $j \ge 1$. 
\end{exe}

\end{document}